\documentclass{amsart}

\usepackage{amsthm}
\usepackage{amsmath}
\usepackage{amssymb}
\usepackage{enumerate}
\usepackage{graphicx}
\usepackage[hidelinks,pagebackref,pdftex]{hyperref}
\usepackage{booktabs}
\usepackage{color}
\usepackage[dvipsnames]{xcolor}
\usepackage{import}
\usepackage{tikz-cd}
\usepackage{microtype}

\renewcommand*{\backref}[1]{}
\renewcommand*{\backrefalt}[4]{
  \ifcase #1
  [No citations.]
  \or [#2]
  \else [#2]
  \fi }

\theoremstyle{plain}
\newtheorem{theorem}{Theorem}[section]
\newtheorem{corollary}[theorem]{Corollary}
\newtheorem{lemma}[theorem]{Lemma}

\newtheorem{proposition}[theorem]{Proposition}

\newtheorem*{namedtheorem}{\theoremname}
\newcommand{\theoremname}{testing}
\newenvironment{named}[1]{\renewcommand{\theoremname}{#1}\begin{namedtheorem}}{\end{namedtheorem}}

\theoremstyle{definition}
\newtheorem{definition}[theorem]{Definition}
\newtheorem{remark}[theorem]{Remark}

\newcommand{\from}{\colon} % As in ``f maps _from_ X _to_ Y''.

\newcommand{\ZZ}{{\mathbb{Z}}}

\newcommand{\QQ}{{\mathbb{Q}}}

\newcommand{\calH}{{\mathcal{H}}}
\newcommand{\calB}{{\mathcal{B}}}
\newcommand{\calT}{{\mathcal{T}}}

\newcommand{\calC}{{\mathcal{C}}}

\newcommand{\refthm}[1]{Theorem~\ref{Thm:#1}}
\newcommand{\reflem}[1]{Lemma~\ref{Lem:#1}}
\newcommand{\refprop}[1]{Proposition~\ref{Prop:#1}}

\newcommand{\refdef}[1]{Definition~\ref{Def:#1}}
\newcommand{\refsec}[1]{Section~\ref{Sec:#1}}
\newcommand{\reffig}[1]{Figure~\ref{Fig:#1}}

\newcommand{\bdy}{\partial}

\newcommand{\PSL}{\operatorname{PSL}}
\newcommand{\SL}{\operatorname{SL}}

\newcommand{\Tr}{\operatorname{Tr}}
\newcommand{\Sp}{\operatorname{Sp}}

\newcommand{\tr}{\operatorname{tr}}

\newcommand{\half}{\frac{1}{2}}
\newcommand{\cut}{\backslash\backslash}

\newcommand{\twist}{\mathrel{%
    \stackrel{\sim}{\smash{\times}\rule{0pt}{0.6ex}}
    }} % twisted product

\title[The triangulation complexity of elliptic and sol 3-manifolds]{The triangulation complexity \\ of elliptic and sol 3-manifolds}

\author{Marc Lackenby}
\address{Mathematical Institute, University of Oxford, Oxford, OX2 6GG, UK}

\author{Jessica S. Purcell}
\address[]{School of Mathematics, Monash University, VIC 3800, Australia }

\begin{document}

\begin{abstract}
The triangulation complexity of a compact 3-manifold $M$ is the minimal number of tetrahedra in
any triangulation of $M$. We compute the triangulation complexity of all elliptic 3-manifolds and all sol 3-manifolds, 
to within a universally bounded multiplicative error.
\end{abstract}

\maketitle

\section{Introduction}\label{Sec:Introduction}

The \emph{triangulation complexity} $\Delta(M)$ of a compact 3-manifold $M$ is the minimal number of tetrahedra in
any triangulation of $M$. (In this paper, we use the definition of a triangulation that has become standard in low-dimensional
topology: it is an expression of $M$ as a union of 3-simplices with some of their faces identified in pairs via affine homeomorphisms.)
Triangulation complexity is a very natural invariant, with some attractive properties. However, its precise value is known for only relatively small examples 
\cite{MartelliPetronio, Matveev:ComplexitySurvey} and for a few infinite families  \cite{MatveevPetronioVesnin, PetronioVesnin, JacoRubinsteinTillmann:LensSpaces, JacoRubinsteinTillmann:Coverings, JacoRubinsteinTillmann:Z2_1, JacoRubinsteinSpreerTillmann:Z2_2, JacoRubinsteinSpreerTillmann:MinimalCusped}. It bears an obvious resemblance to hyperbolic volume, and in fact the volume of a hyperbolic 3-manifold $M$
forms a lower bound for $\Delta(M)$ via the inequality $\Delta(M) \geq \mathrm{Vol}(M)/v_3$, due to Gromov and Thurston
\cite{Thurston:Notes}. Here, $v_3 \simeq 1.01494$ is the volume of a regular hyperbolic ideal tetrahedron.
But non-trivial lower bounds for manifolds with zero Gromov norm have been difficult to obtain.
Jaco, Rubinstein and Tillmann \cite{JacoRubinsteinTillmann:LensSpaces, JacoRubinsteinTillmann:Coverings} were able to compute the triangulation complexity of lens spaces
of the form $L(2n,1)$ and $L(4n,2n \pm1)$. However, general lens spaces have remained out of reach. In this paper,
we remedy this, by computing  the triangulation complexity of all elliptic 3-manifolds and all sol 3-manifolds, 
to within a universally bounded multiplicative error.
Our result about lens spaces confirms a conjecture of Jaco and Rubinstein \cite{JacoRubinstein:Layered} and Matveev \cite{Matveev:Complexity6,Matveev:ComplexitySurvey}, up to a bounded multiplicative constant.

\begin{theorem}\label{Thm:LensSpaces}
Let $L(p,q)$ be a lens space, where $p$ and $q$ are coprime integers satisfying $0< q < p$. Let $[a_0, \dots, a_n]$ be the continued fraction expansion of $p/q$ where each $a_i $ is positive. Then there is a universal constant $k_{\mathrm{lens}}>0$ such that
\[ k_{\mathrm{lens}} \sum_{i=0}^n a_i \leq \Delta(L(p,q)) \leq \sum_{i=0}^n a_i. \]
\end{theorem}

General elliptic 3-manifolds fall into three categories: lens spaces, prism manifolds and a third class that we call Platonic manifolds; for example see Scott~\cite{Scott} for a discussion of the classification of these manifolds.
Recall that the prism manifold $P(p,q)$ is obtained from the orientable $I$-bundle over the Klein bottle, by attaching a solid torus, so that the meridian of the solid torus is identified with the $p/q$ curve on the boundary torus. Here, a canonical framing of this boundary torus is used, so that the longitude and meridian are lifts of non-separating simple closed curves on the Klein bottle that are, respectively, orientation-reversing and orientation-preserving.

\begin{theorem} \label{Thm:Prism}
Let $p$ and $q$ be non-zero coprime integers and let $[a_0, \dots, a_n]$ denote the continued fraction expansion of $p/q$ where $a_i$ is positive for each $i > 0$. Then, $\Delta(P(p,q))$ is, to within a universally bounded multiplicative error, equal to $\sum_{i=0}^n a_i$.
\end{theorem}

We say that an elliptic 3-manifold is \emph{Platonic} if it admits a Seifert fibration where the base orbifold is the quotient of $S^2$ by the orientation-preserving symmetry group of a Platonic solid. These orbifolds have underlying space the 2-sphere and have three exceptional points with orders $(2,3,3)$, $(2,3,4)$ or $(2,3,5)$. The Seifert fibration is specified by the Seifert data, which describes the three singular fibres and includes the Euler number of the fibration. It turns out that the latter quantity controls the triangulation complexity.

\begin{theorem}\label{Thm:Platonic}
Let $M$ be a Platonic elliptic 3-manifold, and let $e$ denote the Euler number of its Seifert fibration. Then, to within a universally bounded multiplicative error, $\Delta(M)$ is $|e|$.
\end{theorem}

We also examine sol 3-manifolds. Recall that these are 3-manifolds of the form  $(T^2 \times [0,1]) / (A(x,1) \sim (x,0))$ where $A$ is an element of $\mathrm{SL}(2, \mathbb{Z})$ with $|\mathrm{tr}(A)| > 2$. Such a matrix $A$ induces a homeomorphism of the torus that is known as \emph{linear Anosov}. Let $\overline{A}$ be the image of $A$ in $\mathrm{PSL}(2, \mathbb{Z})$. Recall that $\mathrm{PSL}(2, \mathbb{Z})$ is isomorphic to $\mathbb{Z}_2 \ast \mathbb{Z}_3$ where the factors are generated by 
\[
S = 
\left(
\begin{matrix}
0 & -1 \\
1 & 0 \\
\end{matrix}
\right)
\qquad
T = 
\left(
\begin{matrix}
0 & -1 \\
1 & -1 \\
\end{matrix}
\right).
\]
Thus any element of $\mathrm{PSL}(2, \mathbb{Z})$ can be written uniquely as a word that is an alternating product of elements $S$ and $T$ or $T^{-1}$. The word is \emph{cyclically reduced} if the first letter is neither the inverse of the final letter nor equal to the final letter. Any element of $\mathrm{PSL}(2, \mathbb{Z})$ is conjugate to a cyclically reduced word that is unique up to cyclic permutation. Our first theorem about sol manifolds relates the triangulation complexity of the manifold to the length of this cyclically reduced word.

\begin{theorem}\label{Thm:SolAlternative}
Let $A$ be an element of $\mathrm{SL}(2, \mathbb{Z})$ with $|\mathrm{tr}(A)| > 2$. Let $M$ be the sol 3-manifold $(T^2 \times [0,1]) / (A(x,1) \sim (x,0))$. Let $\overline{A}$ be the image of $A$ in 
$\mathrm{PSL}(2, \mathbb{Z})$ and let $\ell(\overline{A})$ be the length of a cyclically reduced 
word in the generators $S$ and $T^{\pm 1}$ that is conjugate to $\overline{A}$. Then, there is a universal constant $k_{\mathrm{sol}}>0$ such that
\[ k_{\mathrm{sol}} \ell(\overline{A}) \leq \Delta(M) \leq  (\ell(\overline{A})/2) + 6.\]
\end{theorem}

Note that this length $\ell(\overline{A})$ is readily calculable. For we may simplify $A$ using row operations until it is the identity matrix. This writes $A$ as a product of elementary matrices. The image of each elementary matrix in $\mathrm{PSL}(2, \mathbb{Z})$ is a word in the generators $S$ and $T$. Thus, we obtain $\overline{A}$ as a word in $S$ and $T$. If the starting letter is equal to the inverse of the final letter or equal to the final letter, then we may conjugate by the inverse of this element to create a shorter word.
Thus, eventually, we end with a cyclically reduced word, and $\ell(\overline{A})$ is its length.

Our second theorem relates the triangulation complexity of the 3-manifold $M$ to the continued fraction expansion of $\sqrt{\mathrm{tr}(A)^2 - 4}$. As $\mathrm{tr}(A)^2 - 4$ is not a perfect square, the continued fraction expansion of $\sqrt{\mathrm{tr}(A)^2 - 4}$ does not terminate. Denote it by $[a_0, a_1, \dots]$. As $\sqrt{\mathrm{tr}(A)^2 - 4}$ is the square root of a positive integer, the continued fraction expansion is eventually periodic, in the sense that for some non-negative integer $r$ and even positive integer $t$, $a_{i+t} = a_i$ for every $i \geq r$. The \emph{periodic part} of the continued fraction expansion is $(a_r, \dots, a_{r+t-1})$, which is well-defined up to cyclic permutation.

\begin{theorem}\label{Thm:Sol}
Let $A$ be an element of $\mathrm{SL}(2, \mathbb{Z})$ with $|\mathrm{tr}(A)| > 2$. Let $\overline{A}$ be the image of $A$ in $\mathrm{PSL}(2, \mathbb{Z})$. Suppose that $\overline{A}$ is $B^n$ for some positive integer $n$ and some $B \in \mathrm{PSL}(2, \mathbb{Z})$ that cannot be expressed as a proper power. Let $M$ be the sol 3-manifold $(T^2 \times [0,1]) / (A(x,1) \sim (x,0))$. Let $[a_0, a_1, \dots]$ be the continued fraction expansion of $\sqrt{\mathrm{tr}(A)^2 - 4}$ where $a_i$ is positive for each $i >0$ and let $(a_r, \dots, a_s)$ denote its periodic part. Then there is a universal constant $k'_{\mathrm{sol}}>0$ such that
\[ k'_{\mathrm{sol}} n \sum_{i=r}^s a_i \leq \Delta(M) \leq 6 + n \sum_{i=r}^s a_i. \]
\end{theorem}

Crucial to our arguments is the analysis of triangulations of $T^2 \times [0,1]$. This is because $T^2 \times [0,1]$ arises when we cut a sol manifold along a torus fibre, or when we remove the core curves of the Heegaard solid tori of a lens space. Our result about these products is as follows.

\begin{theorem}\label{Thm:ProductComplexity}
Let $\mathcal{T}_0$ and $\mathcal{T}_1$ be 1-vertex triangulations of the torus $T^2$. Let $\Delta(\mathcal{T}_0, \mathcal{T}_1)$ denote the minimal number of tetrahedra in any triangulation of $T^2 \times [0,1]$ that equals $\mathcal{T}_0$ on $T^2 \times \{ 0 \}$ and equals $\mathcal{T}_1$ on $T^2 \times \{ 1 \}$. Then there is a universal constant $k_{\mathrm{prod}} >0$ such that
\[ k_{\mathrm{prod}} \ d_{\mathrm{Tr}(T^2)}(\mathcal{T}_0, \mathcal{T}_1)  \leq \Delta(\mathcal{T}_0, \mathcal{T}_1) \leq d_{\mathrm{Tr}(T^2)}(\mathcal{T}_0, \mathcal{T}_1) + 6. \]
\end{theorem}

Here, $\mathrm{Tr}(S)$ denotes the triangulation graph for a closed orientable surface $S$, defined to have a vertex for each isotopy class of 1-vertex triangulation of $S$, and where two vertices are joined by an edge if and only if the corresponding triangulations differ by a 2-2 Pachner move. Each edge is declared to have length $1$, and this induces the metric $d_{\mathrm{Tr}(S)}( \, \cdot \, , \, \cdot \, )$. When $S$ is the torus, this graph is in fact equal to the classical Farey tree and so distances in the graph can readily be computed using continued fractions.

This paper is a continuation of the work in \cite{LackenbyPurcell:Fibred}, where we analysed the triangulation complexity of 3-manifolds $M$ that fibre over the circle with fibre a closed orientable surface $S$ with genus at least $2$. We were able to estimate $\Delta(M)$, to within a bounded factor depending only on the genus of $S$, in the case where the monodromy $\phi$ of the fibration is pseudo-Anosov. Our theorem related $\Delta(M)$ to the translation length of the action of $\phi$ on various metric spaces.

Recall that if $X$ is a space with metric $d$, and $\phi$ is an isometry of $X$, its \emph{translation length} $\ell_X(\phi)$ is $\inf \{ d(x,\phi(x)) : x \in X \}$. Its \emph{stable translation length} $\overline{\ell}_X(\phi)$ is $\inf \{ d(\phi^n(x), x)/n : n \in \mathbb{Z}_{>0} \}$, where $x \in X$ is chosen arbitrarily.

Each homeomorphism $\phi$ of $S$ naturally induces an isometry of $\mathrm{Tr}(S)$. It also induces an isometry of the mapping class group $\mathrm{MCG}(S)$, where $\mathrm{MCG}(S)$ is given a word metric by making a fixed choice of some finite generating set. The homeomorphism $\phi$ also acts isometrically on Teichm\"uller space, with its Teichm\"uller or Weil-Petersson metrics. The \emph{thick} part of Teichm\"uller space consists of those hyperbolic structures where every geodesic on the surface has length at least some suitably chosen $\epsilon > 0$. If $\epsilon > 0$ is sufficiently small, the thick part is path connected, and so may be given its path metric. The homeomorphism $\phi$ also induces an isometry of these metric spaces.

The following was the main theorem of \cite{LackenbyPurcell:Fibred}.
The statement below combines the statement of \cite[Theorem~1.3]{LackenbyPurcell:Fibred} with Theorems~3.5 and Proposition~2.7 of that paper. 

\begin{theorem}\label{Thm:Fibred}
Let $S$ be a closed orientable connected surface with genus at least $2$, and let $\phi\from S\to S$ be a pseudo-Anosov homeomorphism. Then the following quantities are within bounded ratios of each other, where the bounds depend only on the genus of $S$ and a choice of finite generating set for $\mathrm{MCG}(S)$:
\begin{enumerate}
\item the triangulation complexity of $(S \times I)/ \phi$;
\item the translation length (or stable translation length) of $\phi$ in the thick part of the Teichm\"uller space of $S$;
\item the translation length (or stable translation length) of $\phi$ in the mapping class group of $S$;
\item the translation length (or stable translation length) of $\phi$ in $\mathrm{Tr}(S)$.
\end{enumerate}
\end{theorem}

We also analysed products.

\begin{theorem}[Theorem~1.4 of \cite{LackenbyPurcell:Fibred}]\label{Thm:TriangulationProductOneVertex}
Let $S$ be a closed orientable surface with genus at least $2$ and let $\mathcal{T}_0$ and $\mathcal{T}_1$ be non-isotopic 1-vertex triangulations of $S$. Then the following are within a bounded ratio of each other, the bounds only depending on the genus of $S$:
\begin{enumerate}
\item the minimal number of tetrahedra in any triangulation of $S \times [0,1]$ that equals $\mathcal{T}_0$ and $\mathcal{T}_1$ on $S \times \{ 0 \}$ and $S \times \{ 1 \}$ respectively;
\item the minimal number of 2-2 Pachner moves relating $\mathcal{T}_0$ and $\mathcal{T}_1$.
\end{enumerate}
\end{theorem}

Theorems~\ref{Thm:SolAlternative} and~\ref{Thm:ProductComplexity} are natural generalisations of these results to the case where $S$ is a torus. For technical reasons, we were unable to deal with this case in \cite{LackenbyPurcell:Fibred}. In this paper, we \emph{deduce} Theorems~\ref{Thm:SolAlternative} and \ref{Thm:ProductComplexity} from \refthm{TriangulationProductOneVertex}, by passing to a branched covering space.

We record here the analogue of \refthm{Fibred} for sol manifolds.

\begin{theorem}\label{Thm:FibredTori}
Let $\phi\from T^2 \to T^2$ be a linear Anosov homeomorphism. Then the following quantities are within universally bounded ratios of each other:
\begin{enumerate}
\item the triangulation complexity of $(T^2 \times I)/ \phi$;
\item the translation length (or stable translation length) of $\phi$ in the thick part of the Teichm\"uller space of $T^2$;
\item the translation length (or stable translation length) of $\phi$ in the mapping class group of $T^2$;
\item the translation length (or stable translation length) of $\phi$ in $\mathrm{Tr}(T^2)$.
\end{enumerate}
In (3), we metrise $\mathrm{MCG}(T^2)$ by fixing the finite generating set \[
\left(
\begin{matrix}
1 & 1 \\
0 & 1 \\
\end{matrix}
\right),
\quad
\left(
\begin{matrix}
1 & 0 \\
1 & 1 \\
\end{matrix}
\right).
\]
\end{theorem}

The structure of the paper is as follows. In Section \ref{Sec:HandleStructures}, we recall some basic facts about handle structures, including the notion of a parallelity bundle. Section \ref{Sec:VerticalArcs} contains the first substantial new result, \refthm{VerticalArc}. This asserts that for any triangulation $\calT$ of $S \times I$, the product of a closed orientable surface $S$ and an interval, there is an arc isotopic to $\ast \times I$, for some point $\ast \in S$, that is simplicial in $\calT^{(23)}$, the 23rd iterated barycentric subdivision. 
This is technically important, because it allows us to transfer a triangulation of $S \times I$ to a triangulation of a suitable branched cover. This is used later in Section \ref{Sec:BrancherCover}, where \refthm{ProductComplexity} is proved using \refthm{TriangulationProductOneVertex}. A suitable finite branched cover of the torus $T^2$ over one point is used, which is a closed orientable surface $S$ of genus greater than one. In order to compare translation lengths in $\mathrm{Tr}(S)$ and $\mathrm{Tr}(T^2)$, we develop some background theory in Sections \ref{Sec:Spines}, \ref{Sec:TriangulationsTorus}, \ref{Sec:TrainTrack} and \ref{Sec:HomeoTorus}.
In Section  \ref{Sec:Spines}, we introduce $\mathrm{Sp}(S)$, which is the space of spines for a closed orientable surface $S$. There is a quasi-isometry between $\mathrm{Tr}(S)$ and $\mathrm{Sp}(S)$ that is equivariant under the action of the mapping class group, but it is useful to consider both spaces. In Section \ref{Sec:TriangulationsTorus}, we recall the relationship between $\mathrm{Tr}(T^2)$, the Farey graph and continued fractions. In Section \ref{Sec:TrainTrack}, we recall results of Masur, Mosher and Schleimer \cite{MasurMosherSchleimer}, which use train tracks to estimate distances in $\mathrm{Tr}(S)$. In Section \ref{Sec:HomeoTorus}, this is applied specifically in the case of the torus. Finally in Sections \ref{Sec:Products},  \ref{Sec:Sol}, \ref{Sec:Lens} and \ref{Sec:Prism}, we deal with products, sol manifolds, lens spaces, prism manifolds and Platonic manifolds.

%% \subsection{Acknowledgements}
%% Purcell was partially supported by the Australian Research Council, grant DP210103136. 

%%%%%%%%%%%%%%%%%%%%%%%%%%%%%%%%%%%%%%%%%%%%%%%%%%%%%%%%%%%%%%%%%
\section{Handle structures and parallelity bundles}
\label{Sec:HandleStructures}

Although the main results of this paper are on the complexity of triangulations of 3-manifolds, for our arguments it is often more convenient to work with handle structures, similarly to \cite{LackenbyPurcell:Fibred}. This section collects some of the definitions and results on handle structures that we will use. 

Recall that a handle structure on a 3-manifold is a decomposition into $i$-handles $D^i\times D^{3-i}$, $i=0,1,2,3$, where we require:
\begin{enumerate}
\item Each $i$-handle intersects the handles of lower index in $\bdy D^i \times D^{3-i}$.
\item Any two $i$-handles are disjoint.
\item The intersection of any 1-handle with any 2-handle is of the form:
  \begin{itemize}
  \item $D^1\times \alpha$ in the 1-handle $D^1\times D^2$, where $\alpha$ is a collection of arcs in $\bdy D^2$,
  \item $\beta\times D^1$ in the 2-handle $D^2\times D^1$, where $\beta$ is a collection of arcs in $\bdy D^2$.
  \end{itemize}
\item Any 2-handle runs over at least one 1-handle.
\end{enumerate}

For example, given a triangulation of a 3-manifold $M$, there is an associated handle structure for $M$ minus an open collar neighbourhood of $\partial M$, called the \emph{dual handle structure}. 
This has 0-handles obtained by removing a thin regular open neighbourhood of the boundary of each tetrahedron, 1-handles obtained by taking a neighbourhood of each face not in $\partial M$ and removing a thin regular open neighbourhood of each edge, 2-handles obtained by taking a neighbourhood of each edge not in $\partial M$ with neighbourhoods of endpoints removed, and 3-handles consisting of a regular neighbourhood of each vertex not in $\partial M$.

One feature of a handle structure that does not hold for a triangulation is that it can be sliced along a normal surface to yield a new handle structure, whereas cutting a tetrahedron along a normal surface does not yield pieces that are tetrahedra, in general. Slicing along normal surfaces in this manner is important for our arguments; the definition below will help us investigate the handle structures that arise.

\begin{definition} A handle structure of a 3-manifold $M$ is \emph{pre-tetrahedral} if the intersection between each 0-handle and the union of the 1-handles and 2-handles is one of the following possibilities:
\begin{enumerate}
\item \emph{tetrahedral}, as shown in the far left of Figure~\ref{Fig:PreTetrahedral};
\item \emph{semi-tetrahedral}, as shown in the middle left of Figure~\ref{Fig:PreTetrahedral};
\item \emph{a product annulus of length 3}, as shown in the middle right of Figure~\ref{Fig:PreTetrahedral};
\item \emph{a parallelity annulus of length 4}, as shown in the far right of Figure~\ref{Fig:PreTetrahedral}.
\end{enumerate}
In that figure, the shaded regions denote discs that can be components of intersection between the 0-handle and $\partial M$, or between the 0-handle and a 3-handle. The hashed regions denote components of intersection between the 0-handle and $\partial M$.
\end{definition}

\begin{figure}
  \includegraphics{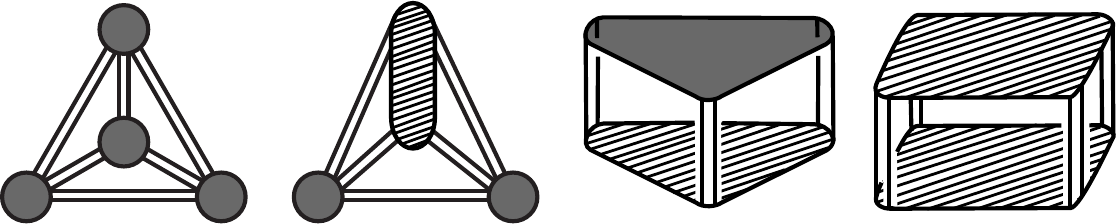}
  \caption{Far left: tetrahedral, Middle left: semi-tetrahedral, Middle right: product annulus of length 3, Far right: Parallelity annulus of length 4}
  \label{Fig:PreTetrahedral}
\end{figure}

\begin{lemma}\label{Lem:InheritedHandleStruct}
Let $\calT$ be a triangulation of a compact orientable 3-manifold $M$. Let $S$ be a normal surface properly embedded in $M$. Then $M \cut S$ inherits a handle structure. Moreover when $M$ is closed, the handle structure on $M \cut S$ is pre-tetrahedral.
\end{lemma}

\begin{proof}
This is explained in \cite[Lemma~4.4]{LackenbyPurcell:Fibred} and so we just sketch the argument. We first form the handle structure $\calH$ dual to $\calT$. The normal surface $S$ intersects each $i$-handle in discs. When the $i$-handle is cut along these discs, the result is a collection of $i$-handles in the required handle structure for $M \cut S$.
\end{proof}

We will measure the size of a triangulation of a 3-manifold using the following quantity.

\begin{definition}
The \emph{complexity} of a triangulation $\calT$ of a compact 3-manifold is the number of tetrahedra of $\calT$ and is denoted $\Delta(\calT)$.
\end{definition}

The corresponding definition for a pre-tetrahedral handle structure is somewhat more complicated. 

\begin{definition}\label{Def:ComplexHS}
Let $H_0$ be a 0-handle of a pre-tetrahedral handle structure $\calH$. Let $\alpha$ be the number of components of intersection between $H_0$ and the 3-handles. Define $\beta$ as follows:
\begin{enumerate}
\item $\beta = 1/2$ if $H_0$ is tetrahedral;
\item $\beta = 1/4$ if $H_0$ is semi-tetrahedral;
\item $\beta = 0$ if $H_0$ is a product or parallelity annulus.
\end{enumerate}
Define the \emph{complexity} of $H_0$ to be $ (\alpha/8) + \beta$.
Define the \emph{complexity} $\Delta(\calH)$ of $\calH$ to be the sum of the complexities of its 0-handles.
\end{definition}

The motivation for this definition is from the following result \cite[Lemma~4.12]{LackenbyPurcell:Fibred}.

\begin{lemma} 
\label{Lem:ComplexityUnderCutting}
Let $\calT$ be a triangulation of a closed orientable 3-manifold $M$. Let $S$ be a normal surface embedded in $M$. Then the handle structure $\calH$ that $M \cut S$ inherits, as in \reflem{InheritedHandleStruct}, satisfies $\Delta(\calH) = \Delta(\calT)$.
\end{lemma}

\begin{lemma}
\label{Lem:ComplexityDualHandleStructure}
Let $\calT$ be a triangulation of a compact orientable 3-manifold $M$. Suppose that the intersection between any tetrahedron $T$ of $\calT$ and $\partial M$ is either empty, a vertex of $T$, an edge of $T$ or a face of $T$. Then dual to $\calT$ is a pre-tetrahedral handle structure $\calH$ of $M$ satisfying $\Delta(\calH) \leq \Delta(\calT)$.
\end{lemma}

\begin{proof}
The dual handle structure $\calH$ has an $i$-handle for each $(3-i)$-simplex of $\calT$ that does not lie wholly in $\partial M$. When the intersection between a tetrahedron and $\partial M$ is empty or a vertex, the dual 0-handle is tetrahedral. When the intersection between a tetrahedron and $\partial M$ is an edge, the dual 0-handle is semi-tetrahedral. When the intersection between a tetrahedron and $\partial M$ is a face, the dual 0-handle is a product annulus of length 3. Since the number of 0-handles of $\calH$ is at most the number of tetrahedra of $\calT$, and each 0-handle contributes at most $1$ to $\Delta(\calH)$, we deduce that $\Delta(\calH) \leq \Delta(\calT)$.
\end{proof}

\begin{remark}
\label{Rem:AddBoundaryTimesI}
Given any triangulation $\calT'$ of a compact orientable 3-manifold $M$, we can form a triangulation $\calT$ of $M$ satisfying the hypotheses of Lemma \ref{Lem:ComplexityDualHandleStructure} with $\Delta(\calT) \leq 33 \Delta(\calT')$. To do this, we attach a triangulation of $\partial M \times I$ to $\calT'$. This triangulation is formed as follows. For each triangle of $\calT'$ in $\partial M$, form its product with $I$, which is a prism. Subdivide each of its square faces into two triangles. Then triangulate each prism by coning from a new vertex in its interior. Each prism is triangulated using $8$ tetrahedra. Since the number of triangles of $\calT'$ in $\partial M$ is at most $4 \Delta(\calT')$, the resulting triangulation $\calT$ satisfies $\Delta(\calT) \leq 33 \Delta(\calT')$.
\end{remark}

\begin{definition} 
Let $M$ be a compact 3-manifold with a handle structure $\calH$, and let $S$ be a subsurface of $\partial M$.
When $\bdy M$ meets 0-handles in discs, $\bdy M$ inherits a handle structure from $\calH$: an $i$-handle of $\bdy M$ is a component of intersection of $\bdy M$ with an $(i+1)$-handle of $\calH$.
We say that $\calH$ is a \emph{handle structure for the pair} $(M,S)$ if the following all hold:
\begin{enumerate}
\item $\partial S$ intersects each handle of $\bdy M$ in a collection of arcs;
\item $\partial S$ misses the 2-handles of $\bdy M$;
\item $\partial S$ respects the product structure of the 1-handles of $\bdy M$.
\end{enumerate}
\end{definition}

\begin{definition}
Let $\calH$ be a handle structure for the pair $(M,S)$.
A handle $H$ of $\calH$ is a \emph{parallelity handle} if it has a product structure $D^2 \times I$ such that
\begin{enumerate}
\item $D^2 \times \partial I = H \cap S$;
\item each component of intersection between $H$ and any other handle is of the form $\beta \times I$ for some subset $\beta$ of $\partial D^2$.
\end{enumerate}
\end{definition}

For example, a product annulus of length 3 meeting $S$ on its top and bottom, and a parallelity annulus of length 4 are both parallelity 0-handles. There will also be parallelity 1-handles and 2-handles. 

\begin{remark}
\label{Rem:NoParallelity}
The handle structure $\calH$ in Lemma \ref{Lem:ComplexityDualHandleStructure} has no parallelity handles when viewed as a handle structure for the pair $(M, \partial M)$. To see this, note that any parallelity handle for $(M, \partial M)$ is adjacent to a parallelity 2-handle.
However, a parallelity 2-handle of $\calH$ is dual to an edge of $\calT$ with both endpoints in $\partial M$ but with interior in the interior of $M$. However, if there were such an edge of $\calT$, then the intersection between any adjacent tetrahedron and $\partial M$ would violate the hypothesis of Lemma \ref{Lem:ComplexityDualHandleStructure}.
\end{remark}

\begin{definition}
The \emph{parallelity bundle} for $\calH$ is the union of the parallelity handles. 
\end{definition}

It was shown in \cite[Lemma~3.3]{Lackenby:CrossingNumberComposite} that the $I$-bundle structures on the parallelity handles can be chosen to patch together to form an $I$-bundle structure on the parallelity bundle. We therefore use the following standard terminology.

\begin{definition}
Let $B$ be an $I$-bundle over a surface $F$. Its \emph{horizontal boundary} $\partial_h B$ is the $(\partial I)$-bundle over $F$. Its \emph{vertical boundary} $\partial_vB$ is the $I$-bundle over $\partial F$.
\end{definition}

It is often very useful to enlarge the parallelity bundle, forming the following structure.

\begin{definition}\label{Def:GeneralisedParallelityBundle}
Let $M$ be a compact orientable 3-manifold and let $S$ be a subsurface of $\partial M$. Let $\mathcal{H}$ be a handle structure for $(M,S)$. 
A \emph{generalised parallelity bundle} $\mathcal{B}$ is a 3-dimensional submanifold of $M$ such that
\begin{enumerate}
\item $\mathcal{B}$ is an $I$-bundle over a compact surface;
\item the horizontal boundary $\bdy_h\mathcal{B}$ of $\mathcal{B}$ is the intersection between $\mathcal{B}$ and $S$;
\item $\mathcal{B}$ is a union of handles of $\mathcal{H}$;
\item any handle of $\mathcal{B}$ that intersects $\partial_v \mathcal{B}$ is a parallelity handle, where the $I$-bundle structure on the parallelity handle agrees with the $I$-bundle structure of $\mathcal{B}$;
\item whenever a handle of $\mathcal{H}$ lies in $\mathcal{B}$ then so do all incident handles of $\mathcal{H}$ with higher index;
\item the intersection between $\partial_h \mathcal{B}$ and the non-parallelity handles lies in a union of disjoint discs in the interior of $S$.
\end{enumerate}
\end{definition}

Note that condition (6) is included in the definition given in \cite{LackenbyPurcell:Fibred} but is not in some earlier work  \cite{Lackenby:CrossingNumberComposite}.

The main reason why this is such a useful notion is the fact that frequently we may ensure that the horizontal boundary of a generalised parallelity bundle is incompressible.

\begin{definition}\label{Def:AnnularSimplification}
Let $M$ be a compact orientable irreducible 3-manifold and let $S$ be a subsurface of $\partial M$. Let $\mathcal{H}$ be a handle structure for $(M,S)$.
Suppose $M$ contains the following:
\begin{enumerate}
\item an annulus $A'$ that is a vertical boundary component of a generalised parallelity bundle $\calB$;
\item an annulus $A$ contained in $S$ such that $\partial A=\partial A'$;
\item a 3-manifold $P$ with $\partial P = A \cup A'$ such that $P$ either lies in a 3-ball or is a product region between $A$ and $A'$.
\end{enumerate}
Suppose also that $P$ is a union of handles of $\calH$, that whenever a handle of $\calH$ lies in $P$, so do all incident handles with higher index, and that any parallelity handle of $\calH$ that intersects $P$ lies in $P$. Finally, suppose that apart from the component of the generalised parallelity bundle incident to $A'$, all other components of $\mathcal{B}$ in $P$ are $I$-bundles over discs.

An \emph{annular simplification} of the 3-manifold $M$ is the manifold obtained by removing the interiors of $P$ and $A$ from $M$; see \reffig{AnnularSimp}. 
\end{definition}

\begin{figure}
  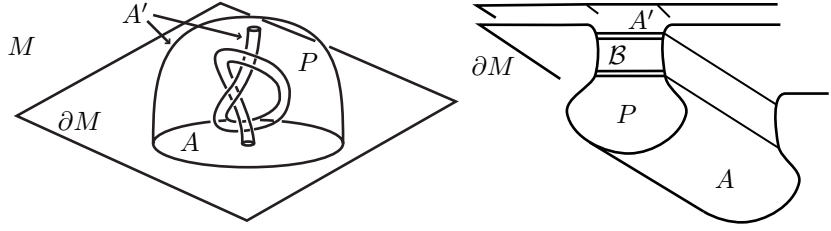
  \caption{An annular simplification removes $P$ and replaces $A$ with $A'$. Left: an annular simplification where $P$ lies within a ball. Right: a cross-section of an annular simplification where $P$ is a product region between $A$ and $A'$.}
  \label{Fig:AnnularSimp}
\end{figure}

\begin{lemma}\label{Lem:GenParIncompr}
  Let $M$ be a compact orientable irreducible 3-manifold and let $S$ be an incompressible subsurface of $\partial M$ that is not a 2-sphere. Let $\mathcal{H}$ be a handle structure for $(M,S)$. Let $\calB$ be a generalised parallelity bundle that is maximal, in the sense that it is not a proper subset of another generalised parallelity bundle. Suppose that $\mathcal{H}$ admits no annular simplification. Then $\mathcal{B}$ contains every parallelity handle of $\mathcal{H}$, and moreover, each component of $\mathcal{B}$:
  \begin{enumerate}
  \item has incompressible horizontal boundary, and
  \item either has incompressible vertical boundary, or is an $I$-bundle over a disc. 
  \end{enumerate}
\end{lemma}

\begin{proof}
This is stated in \cite[Corollary 5.7]{Lackenby:CrossingNumberComposite}. However, a slightly different definition of generalised parallelity bundle is used there that omits Condition 6 in Definition~\ref{Def:GeneralisedParallelityBundle}. This extra condition does not affect the argument there.

Alternatively, one can argue as follows.  In \cite[Theorem~6.18]{LackenbyPurcell:Fibred}, we showed that $\mathcal{B}$ contains every parallelity handle of $\mathcal{H}$, and every component either satisfies the conclusion of the lemma, or is a special case called boundary-trivial. In the boundary-trivial case, the component of $\mathcal{B}$ lies within a 3-ball; the precise definition is \cite[Definition~6.16]{LackenbyPurcell:Fibred}. However, \cite[Lemma~6.17]{LackenbyPurcell:Fibred} implies that a boundary-trivial component admits an annular simplification. Thus we cannot have such components by hypothesis. 
\end{proof}

The \emph{weight} of a surface properly embedded in a manifold $M$, in general position with respect to a triangulation $\calT$, is defined to be the number of intersections between $S$ and the edges of $\calT$. 

The following is \cite[Lemma~6.15]{LackenbyPurcell:Fibred}.

\begin{theorem} 
\label{Thm:ExtendToGenParBdle}
Let $\calT$ be a triangulation of a compact orientable irreducible 3-manifold $M$. Let $S$ be an orientable incompressible normal surface properly embedded in $M$ that has least weight, up to isotopy supported in the interior of $M$. Let $\calH$ be the handle structure that $M' = M \cut S$ inherits, as in \reflem{InheritedHandleStruct}. Let $S' = \partial M' \cut \partial M$. Then $(M',S')$ admits no annular simplification. Hence, the parallelity bundle for $\calH$ extends to a maximal generalised parallelity bundle that has incompressible horizontal boundary.
\end{theorem}

The following is \cite[Lemma~8.14]{LackenbyPurcell:Fibred}.

\begin{lemma}
\label{Lem:LengthVerticalBoundary}
Let $\calH$ be a pre-tetrahedral handle structure of a pair $(M,S)$, and let $\calB$ be its parallelity bundle. Then the length of $\partial_v \calB$, which is its number of 2-cells, is at most $56 \Delta(\calH)$.
Similarly if $\calB'$ is a maximal generalised parallelity bundle, then the length of $\bdy_v\calB'$ is at most $56\Delta(\calH)$. 
\end{lemma}

\begin{definition}
Let $\calH$ be a handle structure of a compact 3-manifold. Then the \emph{associated cell structure} is obtained as follows:
\begin{enumerate}
\item each handle is a 3-cell;
\item each component of intersection between two handles or between a handle and $\partial M$ is a 2-cell;
\item each component of intersection between three handles or between two handles and $\partial M$ is a 1-cell;
\item each component of intersection between four handles or between three handles and $\partial M$ is a 0-cell.
\end{enumerate}
\end{definition}

\begin{lemma}
\label{Lem:DualToPreTetHS}
Let $\calH$ be a pre-tetrahedral handle structure of a compact orientable 3-manifold $M$. Suppose that $\calH$ has no parallelity 0-handles. Let $\calC$ be the associated cell structure. Let $\calT$ be the triangulation obtained by placing a vertex in the interior of each 2-cell and coning off, and then placing a vertex in the interior of each 3-cell and coning off. Then $\Delta(\calT) \leq 2304 \Delta(\calH)$.
\end{lemma}

\begin{proof}
We first estimate how many triangles there are in the 2-skeleton of $\calC$.
Let $\calH^j$ denote the union of $j$ handles of $\calH$. Thus each 2-cell is a component of $\calH^i \cap \calH^j$ for $i\neq j\in\{0,1, 2\}$ or of $\calH^j\cap (\calH^3\cup\bdy M)$ for $j<3$.

There are as many triangles in $\calH^2 \cap \calH^1$ as in $\calH^2 \cap \calH^0$. Similarly, there are as many triangles in $\calH^2 \cap (\calH^3 \cup \partial M)$ as in $\calH^2 \cap \calH^0$. 
There are as many triangles in $\calH^1 \cap \calH^2$ as in $\calH^1 \cap (\calH^3 \cup \partial M)$. There are as many triangles in $\calH^1 \cap \calH^0$ as in $\calH^1 \cap \calH^2$. There are as many triangles in $\calH^0 \cap (\calH^3 \cup \partial M)$ as in $\calH^0 \cap \calH^2$. Each component of $\calH^2 \cap \calH^0$ is triangulated using 4 triangles. Hence, we see that the total number of triangles in the 2-skeleton of $\calC$ is $24 |\calH^2 \cap \calH^0|$.
Since $\calH$ is pre-tetrahedral, each 0-handle meets at most six 2-handles. Since $\calH$ has no parallelity 0-handles, each 0-handle contributes at least $1/8$ to $\Delta(\calH)$. Thus $24|\calH^2\cap\calH^0|$ is at most $1152\Delta(\calH)$.
Each tetrahedron of $\calT$ has a triangle in $\calC$ as a face, and each triangle in $\calC$ is a face of at most two tetrahedra. Hence, $\Delta(\calT) \leq 2304 \Delta(\calH)$.
\end{proof}

In \refsec{Products}, one of our arguments will replace some semi-tetrahedral 0-handles in a pre-tetrahedral handle structure by 0-handles modified as follows. The boundary of a semi-tetrahedral 0-handle has two 1-handles that are bordered by exactly two 2-handles. We replace the union of one of these 1-handles and the adjacent 2-handles by a single 2-handle. For any semi-tetrahedral 0-handle, this replacement may be done on either one or both of its relevant 1-handles. We call the result a \emph{clipped semi-tetrahedral 0-handle}. One clipped semi-tetrahedral 0-handle is shown in \reffig{ClippedSemiTet}. 

\begin{figure}
  \includegraphics{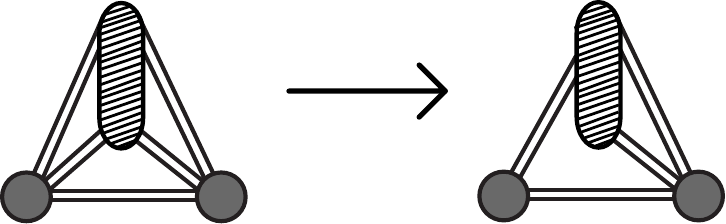}
  \caption{Clipping a semi-tetrahedral 0-handle.}
  \label{Fig:ClippedSemiTet}
\end{figure}

We may define the complexity of a handle structure that is pre-tetrahedral aside from a finite number of clipped semi-tetrahedral 0-handles just as in \refdef{ComplexHS} by setting $\beta$ to be $1/4$ for each clipped semi-tetrahedral 0-handle, and leaving the definition the same otherwise. Then we may modify \reflem{DualToPreTetHS} as follows.

\begin{lemma}
\label{Lem:DualToPreTetHS-Clipped}
Let $\calH$ be a handle structure of a compact orientable 3-manifold $M$. Suppose that $\calH$ is pre-tetrahedral, aside from a finite number of clipped semi-tetrahedral 0-handles. Suppose also that $\calH$ has no parallelity 0-handles. Let $\calC$ be the associated cell structure. Let $\calT$ be the triangulation obtained by placing a vertex in the interior of each 2-cell and coning off, and then placing a vertex in the interior of each 3-cell and coning off. Then $\Delta(\calT) \leq 2304 \Delta(\calH)$.
\end{lemma}

\begin{proof}
  The proof is identical to that of \reflem{DualToPreTetHS}, since a clipped semi-tetrahedral 0-handle still meets at most six 2-handles, and contributes at least $1/4$ to $\Delta(\calH)$. 
\end{proof}

%%%%%%%%%%%%%%%%%%%%%%%%%%%%%%%%%%%%%%%%%%%%%%%%%%%%%%%%%%%%%%%%%
\section{Vertical arcs in products} \label{Sec:VerticalArcs}

As discussed in the introduction, a central part of the paper will be an analysis of triangulations of $S \times [0,1]$, where $S$ is a closed orientable surface. We will want to transfer results about $S \times [0,1]$ to results about $S' \times [0,1]$ where $S'$ is a branched cover of $S$. The branching locus will be an arc of the following form.

\begin{definition} Let $S$ be a closed surface. An arc properly embedded in $S \times [0,1]$ is \emph{vertical} if it is ambient isotopic to $\{ \ast \} \times [0,1]$ for some point $\ast$ in $S$.
\end{definition}

The main result of this section is as follows.

\begin{theorem} \label{Thm:VerticalArc}
Let $S$ be a closed connected orientable surface. Let $\calT$ be a triangulation of $S \times [0,1]$. Then the 23rd iterated barycentric subdivision $\mathcal{T}^{(23)}$ contains an arc in its 1-skeleton that is vertical. 
\end{theorem}

This will be proved using some normal surface theory. The following basic result in the theory is contained in \cite[Proposition~3.3.24, Corollary~3.3.25]{Matveev:Algorithmic}.

\begin{lemma} \label{Lem:IsotopeNormal}
Let $M$ be a compact orientable irreducible 3-manifold with incompressible boundary, and let $\calT$ be a triangulation of $M$. Let $S$ be an incompressible boundary-incompressible surface properly embedded in $M$, no component of which is a sphere or disc, and that is in general position with respect to $\mathcal{T}$.
Then there is an ambient isotopy taking $S$ to a normal surface with weight no greater than that of $S$. Moreover, if $C$ is any boundary curve of $S$ that is normal and intersects each edge of $\calT$ at most once, then the isotopy can be chosen to leave $C$ fixed.
\end{lemma}

\begin{proposition}\label{Prop:ArcAvoidParallelity}
Let $\mathcal{T}$ be a triangulation of a compact 3-manifold $M$. Let $S$ be a 2-sided normal surface properly embedded in $M$. Let $S'$ be the copies of $S$ in $M \cut S$. Let $\mathcal{B}$ be the parallelity bundle for the pair $(M \cut S, S')$. Let $\alpha$ be an arc properly embedded in $M$ with the following properties.
\begin{enumerate}
\item It lies within a copy of $S$ in $M \cut S$.
\item It is disjoint from the horizontal boundary of $\calB$.
\item Its intersection with each normal triangle or square of $S$ is either empty or a single properly embedded arc with endpoints on distinct edges of the triangle or square.
\end{enumerate}
Then $\alpha$ is simplicial in $\mathcal{T}^{(23)}$.
\end{proposition}

\begin{proof}
Within each tetrahedron of $\calT$, the normal discs of $S$ come in at most $5$ types. Let $D$ be the union of the outermost discs of each type. These discs within a single tetrahedron intersect each face of $\calT$ in at most $8$ arcs. However, each face of $\calT$ might be adjacent to two tetrahedra of $\calT$ and there is no reason for the $8$ arcs coming from the two adjacent tetrahedra to coincide. So, the intersection between $D$ and any face of $\calT$ consists of at most $16$ normal arcs. By \cite[Lemma~6.7]{LackenbySchleimer}, the union of the arcs is simplicial in $\calT^{(6)}$. Within each tetrahedron of $\calT$, $D$ consists of at most $10$ normal discs. Hence, by \cite[Lemma~6.11]{LackenbySchleimer}, we may use at most $16$ further subdivisions to make these discs simplicial. 

We now apply one further subdivision to the triangulation, forming $\calT^{(23)}$. We may assume that the intersection between $\alpha$ and the 2-skeleton of $\calT$ is a union of vertices of $\calT^{(23)}$. We may further isotope $\alpha$ so that it is simplicial. This follows from the general result that an arc in triangulated polygon may be isotoped to be simplicial in the barycentric subdivision. Moreover, if the endpoints of the arc are already vertices of this subdivision, then the isotopy can keep these endpoints fixed.
\end{proof}

\begin{proof}[Proof of \refthm{VerticalArc}] 
Suppose first that $S$ is a 2-sphere. Pick any properly embedded simplicial arc in $\calT^{(23)}$ joining $S \times \{0 \}$ to $S \times \{1 \}$. By the lightbulb trick, this is ambient isotopic to an arc of the form $\{ \ast \} \times [0,1]$, as required.

Thus, we may assume that $S$ is not a 2-sphere. Hence, $S \times \{ 0 \}$ contains an essential simple closed curve. Pick one, $C$, that is transverse to the 1-skeleton of $\calT$ on $S\times\{0\}$ and that intersects each edge of that 1-skeleton at most once.
Then $C$ is normal. It is not hard to prove that such a curve must exist; for example, we can take $C$ to be non-trivial in $H_1(S; \mathbb{Z}/2\mathbb{Z})$ and with fewest points of intersection with the edges. Let $A$ be the annulus $C \times [0,1]$. By \reflem{IsotopeNormal}, this can be isotoped, without moving $C \times \{ 0 \}$, to a normal surface. We pick $A$ to have least weight among all annuli with one boundary component equal to $C \times \{ 0 \}$ and the other boundary component on $S \times \{ 1 \}$. Let $C \times \{ 1 \}$ be the other boundary component of $A$, which is then a normal simple closed curve in $S \times \{ 1 \}$.

In the case where $S$ is a torus, we need to be more precise about the choice of annulus $A$, as follows.
Pick an oriented vertical arc in $S \times [0,1]$ disjoint from $A$. Then the \emph{winding number} of an oriented annulus, with boundary curves equal to $\partial A$, is the signed intersection number of the annulus with this vertical arc. For each winding number $t$, let $w(t)$ be the minimal weight of a normal annulus with boundary equal to $\partial A$ and winding number $t$. If there is no normal annulus with a given winding number $t$, then we define $w(t)$ to be infinite. Note that $w(t)$ tends to infinity as $t \rightarrow \pm \infty$. Now, $A$ has least weight among all normal annuli with the given boundary curves. Hence, $w(0)$ is a global minimum. However, there may be other values of $t$ such that $w(t) = w(0)$. Choose $t_0$ to be maximal with this property. We replace $A$ by a normal annulus, having the same weight and the same boundary curves, but with winding number $t_0$. Call this new annulus $A$.

Let $M$ be the 3-manifold $(S \times [0,1]) \cut A$. Let $\tilde A$ be the two copies of $A$ in $\partial M$. Let $\mathcal{B}$ be the parallelity bundle for the pair $(M, \tilde A)$. This consists of the union of the regions between parallel normal discs of $A$.
By choice of $C$, the curve $C\times \{0\}$ intersects each edge of $\calT$ at most once. Thus no normal disc of $A$ incident to $C \times \{ 0 \}$ is parallel to another normal disc of $A$. Hence, $\calB$ misses $S \times \{ 0 \}$.
By \refthm{ExtendToGenParBdle}, $\mathcal{B}$ extends to a maximal generalised parallelity bundle $\mathcal{B}_+$ that has incompressible horizontal boundary. Its vertical boundary is a union of vertical boundary components of $\calB$, and hence it also misses $S \times \{ 0 \}$.

Since $\partial_h \calB_+$ is an incompressible subsurface of the annuli $\tilde A$, it is a collection of annuli and discs. Hence, each component of $\calB_+$ is an $I$-bundle over a disc, annulus or M\"obius band.

\smallskip

\emph{Claim 1.} No component of $\calB_+$ is an $I$-bundle over a M\"obius band.

The $I$-bundle over a core curve of this M\"obius band would be a M\"obius band embedded in $S \times [0,1]$ with boundary in $\tilde A$. We could then attach an annulus to its boundary, to create a M\"obius band embedded in $S \times [0,1]$ with boundary in $S \times \{ 0 \}$. We could then double $S \times [0,1]$ along $S \times \{ 0 \}$ to create another copy of $S \times [0,1]$ containing a Klein bottle. We could then embed this in the 3-sphere, which is well known to be impossible.

\smallskip

\emph{Claim 2.} No component of $\calB_+$ is an $I$-bundle over an annulus that intersects both components of $\tilde A$.

Let $B$ be such a component. Since $\partial_h B$ consists of incompressible annuli, and because these annuli are disjoint from $S \times \{ 0 \}$, at least one boundary component of $\partial_h B$ is disjoint from $\partial \tilde A$. It is a core curve of a component of $\tilde A$. Let $V$ be the vertical boundary component of $\calB_+$ incident to this core curve; recall $V$ lies in the parallelity bundle $\calB$. Then $V$ is disjoint from $S \times \{ 0,1\}$ and $\partial V$ consists of core curves disjoint from $\partial \tilde A$. Note that $V$ specifies a free homotopy between the two boundary curves of $S \cut C$. Hence, we deduce in this case that $S$ is a torus.

Let $A_1$ and $A_2$ be the two components of $\tilde A$. Each $A_i$ is divided into smaller annuli $A'_i$ and $A''_i$ by $\partial V$, where $A'_i$ is the component intersecting $S \times \{ 0 \}$. We can construct two annuli $A'_1 \cup V \cup A''_2$ and $A'_2 \cup V \cup A''_1$. Using a small isotopy supported in the interior of $S \times [0,1]$, these annuli can be made normal.
Both of these have the same boundary curves as $A$. One has winding number one less than $A$, the other has winding number one more than $A$. The one with winding number greater than $A$ has, by our choice of $A$, weight strictly greater than $A$. But the sum of the weights of $A'_1 \cup V \cup A''_2$ and $A'_2 \cup V \cup A''_1$ is twice the weight of $A$. Hence, the other annulus has weight less than that of $A$. But $A$ was chosen to have minimal weight, which is a contradiction.

\smallskip

\emph{Claim 3.} Each annular component of $\partial_h \calB_+$ is disjoint from $\partial \tilde A$.

Let $B$ be any component of $\calB_+$ that is an $I$-bundle over an annulus. By Claim~2, its two horizontal boundary components both lie in the same component of $\tilde A$. Call this component $A_1$.
Note each component of $\bdy_h B$ contains a core curve of $A_1$, because $\bdy_h B$ is essential. 
Suppose that one component of $\partial_h B$ is an annulus $H$ intersecting $\partial \tilde A$. Because $\calB$ misses $S\times\{0\}$, $H$ must meet $S\times\{1\}$. Then $\partial_v B$, which lies in the parallelity bundle, intersects $S\times\{1\}$. It follows that the other component of $\partial_h B$ also intersects $S\times\{1\}$.
Since $H$ intersects $S \times \{ 1 \}$ by assumption, the other components of $A_1 \cut H$ are an annulus incident to $S \times \{ 0 \}$ and possibly discs incident to $S \times \{ 1 \}$. But the other component of $\partial_h B$ also intersects $S \times \{ 1\}$, and so it must lie in one of these discs. But it cannot then contain a core curve of $A_1$, which is a contradiction.

\smallskip

\emph{Claim 4.} There are no annular components of $\partial_h \calB_+$.

Let $B$ be any component of $\calB_+$ that is an $I$-bundle over an annulus. By Claim~3, $\partial_h B$ is disjoint from $S \times \{ 0,1 \}$ and by Claim~2, it lies in a single component of $\tilde A$. Let $V$ be any vertical boundary component of $B$. Then $\partial V$ cobounds an annulus $A'$ in $A$. If we remove $A'$ from $A$ and replace it by $V$, the result is an annulus with the same boundary as $A$ but with smaller weight. By \reflem{IsotopeNormal}, we may isotope this to a normal annulus without increasing its weight and without moving its intersection curve with $S \times \{ 0 \}$. This contradicts our choice of $A$.

\smallskip

We are now in a position to prove the theorem. Since $\calB_+$ consists only of $I$-bundles over discs, we may find an arc $\alpha$ in $\tilde A$ running from $S \times \{ 0 \}$ to $S \times \{ 1 \}$ and that avoids $\calB_+$. We can choose $\alpha$ with the property that it intersects each triangle or square of $\tilde A$ in a single properly embedded arc with endpoints on distinct edges of the triangle or square. Thus, $\alpha$ satisfies the hypotheses of Proposition \ref{Prop:ArcAvoidParallelity}. It therefore is simplicial in $\calT^{(23)}$. It is the required vertical arc.
\end{proof}

The following lemma will be useful when modifying a given triangulation of $S \times [0,1]$. 

\begin{lemma}
\label{Lem:MaintainVerticalAfterPachner}
Let $\calT$ be a triangulation of $S \times [0,1]$ and let $\alpha$ be a simplicial arc that is vertical in $S \times [0,1]$. Let $\calT'$ be a triangulation obtained from $\calT$ by attaching a tetrahedron to $S \times [0,1]$ to realise a Pachner move of the boundary triangulation. Then $\alpha$ extends to a simplicial arc $\alpha'$ in $\calT'$ that is also vertical.
\end{lemma}

\begin{proof}
The attachment of the tetrahedron realises a Pachner move on the boundary that has type 1-3, 2-2 or 3-1. In the cases of a 1-3 move and a 2-2 move, the arc $\alpha$ remains properly embedded and vertical, and so in these cases, we set $\alpha'$ to be $\alpha$. In the case of a 3-1 Pachner move, the new tetrahedron is incident to three triangles that meet at a vertex. If $\alpha$ does not end at that vertex, then we again set $\alpha'$ to be $\alpha$. If $\alpha$ does end at that vertex, then we form $\alpha'$ by adding one of the edges that is incident to two of the triangles. This is vertical.
\end{proof}

%%%%%%%%%%%%%%%%%%%%%%%%%%%%%%%%%%%%%%%%%%%%%%%%%%%%%%%%%%%%%%%%%
\section{Spines, triangulations and mapping class groups}
\label{Sec:Spines}

In this section, we define a graph associated with a closed orientable surface $S$, the spine graph $\mathrm{Sp}(S)$ on $S$. We show $\mathrm{Sp}(S)$ is quasi-isometric to the triangulation graph $\mathrm{Tr}(S)$ defined in the introduction. We also obtain properties of spines and methods of modifying them that we will use in future arguments. 

Recall the triangulation graph $\mathrm{Tr}(S)$ defined in the introduction. Related to the triangulation graph is the spine graph, defined as follows. 

\begin{definition}
A \emph{spine} for a closed orientable surface $S$ is a graph $\Gamma$ embedded in $S$ that has no vertices of degree $1$ or $2$ and where $S \cut \Gamma$ is a disc.
\end{definition}

\begin{definition}
In an \emph{edge contraction} on a spine $\Gamma$, one collapses an edge that joins distinct vertices, thereby amalgamating these vertices into a single vertex. An \emph{edge expansion} is the reverse of this operation.
\end{definition}

\begin{definition} \label{Def:SpineGraph} 
The \emph{spine graph} $\mathrm{Sp}(S)$ for a closed orientable surface $S$ is a graph defined as follows. It has a vertex for each spine of $S$, up to isotopy of $S$. Two vertices are joined by an edge if and only if their spines differ by an edge contraction or expansion.
\end{definition}

We wish to compare the spine graph and triangulation graph.
Dual to each 1-vertex triangulation is a spine. 
Each 2-2 Pachner move on a 1-vertex triangulation has the following effect on the dual spines: contract an edge and then expand. Thus, each edge in $\mathrm{Tr}(S)$ maps to a concatenation of two edges in $\mathrm{Sp}(S)$. We therefore get a map $\mathrm{Tr}(S) \rightarrow \mathrm{Sp}(S)$.

It will also be useful to recall the following variant of the triangulation graph \cite[Definition~2.6]{LackenbyPurcell:Fibred}.

\begin{definition}\label{Def:TriangulationGraphVertices}
Let $S$ be a closed orientable surface and let $n$ be a positive integer. Then $\mathrm{Tr}(S;n)$ denotes the space of triangulations with at most $n$ vertices. This is a graph with a vertex for each isotopy class of such triangulations, and with an edge for each 2-2, 3-1, or 1-3 Pachner move between them.
\end{definition}

There is an obvious inclusion $\mathrm{Tr}(S) \rightarrow \mathrm{Tr}(S;n)$ for any positive integer $n$. Note also that the mapping class group of $S$ acts on $\mathrm{Tr}(S)$, $\mathrm{Tr}(S;n)$ and $\mathrm{Sp}(S)$ by isometries. Moreover, this action is properly discontinuous and cocompact. Hence, the mapping class group of $S$ is quasi-isometric to each of $\mathrm{Tr}(S)$, $\mathrm{Tr}(S;n)$ and $\mathrm{Sp}(S)$, via an application of the Milnor-\u{S}varc lemma (\cite[Proposition~8.19]{BridsonHaefliger}). In fact, we obtain the following result.

\begin{lemma}\label{Lem:QIMaps}
The maps  $\mathrm{Tr}(S) \rightarrow \mathrm{Tr}(S;n)$ and $\mathrm{Tr}(S) \rightarrow \mathrm{Sp}(S)$ are quasi-isometries.
\end{lemma}

\begin{proof}
Pick a 1-vertex triangulation $T$ for $S$. By the Milnor-\u{S}varc lemma, the map $\mathrm{MCG}(S) \rightarrow \mathrm{Tr}(S)$ sending $g \in \mathrm{MCG}(S)$ to $gT$ is a quasi-isometry; see, for example \cite[Proposition~8.19]{BridsonHaefliger}. A quasi-inverse is given as follows. For $T$ fixed and any point $p$ in $\mathrm{Tr}(S)$, pick a triangulation of the form $gT$ that is closest to $p$. Then the quasi-inverse sends $p$ to $g$. The composition of this quasi-inverse with $\mathrm{MCG}(S) \rightarrow \mathrm{Tr}(S;n)$ is a quasi-isometry  $\mathrm{Tr}(S) \rightarrow \mathrm{Tr}(S;n)$. There is a uniform upper bound to the distance between the image of a triangulation under this map $\mathrm{Tr}(S) \rightarrow \mathrm{Tr}(S;n)$ and its image under the inclusion map. Hence, the inclusion map is also a quasi-isometry as required.

The argument for $\mathrm{Tr}(S) \rightarrow \mathrm{Sp}(S)$ is identical.
\end{proof}

A modification that one can make to a spine that is slightly more substantial than an edge contraction or expansion is as follows.

\begin{definition}\label{Def:EdgeSwap}
Let $\Gamma$ be a spine for a closed surface $S$. Let $e_1$ be an arc properly embedded in the disc $S \cut \Gamma$. Let $e_2$ be an edge of the graph $\Gamma \cup e_1$ that has distinct components of $S \cut (\Gamma \cup e_1)$ on either
side of it. Then the result of removing $e_2$ from $\Gamma$ and adding $e_1$ is a new spine $\Gamma'$ for $S$. We say that $\Gamma$ and $\Gamma'$ are related by an \emph{edge swap}.
\end{definition}

The following is \cite[Lemma~8.3]{LackenbyPurcell:Fibred}.

\begin{lemma}\label{Lem:EdgeSwapBound}
Let $S$ be a closed orientable surface. Let $\Gamma$ be a spine for $S$. Then an edge swap can be realised by a sequence of at most $24g(S)$ edge expansions and contractions.
\end{lemma}

\begin{definition}\label{Def:CellularSpine}
Let $S$ be a closed orientable surface with a cell structure. A spine for $S$ is \emph{cellular} if it is a subcomplex of the 1-skeleton of the cell complex. The \emph{length} of this spine is the number of 1-cells that it contains.
\end{definition}

The following is \cite[Corollary~8.8]{LackenbyPurcell:Fibred}.

\begin{lemma}\label{Lem:SlidingOffDiscs}
Let $S$ be a closed orientable surface with a cell structure $\mathcal{C}$, and with a cellular spine $\Gamma$. Let $D_1, \dots, D_m$ be cellular subsets of $S$, each of which is an embedded disc, and with disjoint interiors.
Let $\ell$ be the sum of the lengths of $\partial D_1, \dots, \partial D_m$. Then there is a sequence of at most $6m g(S) + 2\ell$ edge swaps taking $\Gamma$ to a cellular spine $\Gamma'$ that is disjoint from the interior of $D_1, \dots, D_m$.
\end{lemma}

\begin{remark}
\label{Rem:StrengthenSlidingOffDiscs}
A slight strengthening of the lemma remains true, with the same proof. Instead of $D_1, \dots, D_m$ being embedded discs, we can allow them to be the images of immersed discs in $S$, where the restriction of the immersion to $\mathrm{int}(D_1) \cup \dots \cup \mathrm{int}(D_m)$ is an embedding. In other words, we allow the boundaries of the discs to self-intersect and to intersect each other.
\end{remark}

The following is a version of \reflem{SlidingOffDiscs} dealing with both discs and annuli.

\begin{lemma}\label{Lem:SlidingOffAnnuli}
Let $S$ be a closed orientable surface with a cell structure $\mathcal{C}$, and with a cellular spine $\Gamma$. Let $A_1, \dots, A_m$ be cellular subsets of $S$, each of which is the image of an immersed disc or annulus, and where the restriction of the immersion to the interior of these discs and annuli is an embedding.
Let $\ell$ be the sum of the lengths of $\partial A_1, \dots, \partial A_m$. Then there is a sequence of at most $6m g(S) + 16g(S) + 2m+ 2\ell$ edge swaps taking $\Gamma$ to a cellular spine $\Gamma'$ that is disjoint from the interior of the disc components of $A_1, \dots, A_m$ and that intersects the interior of each annular component in at most one essential embedded arc. Moreover, this arc is a subset of the original spine $\Gamma$.
\end{lemma}

\begin{proof}
In each essential annular component, there must be an essential properly embedded arc that is a subset of $\Gamma$, as otherwise the disc $S \cut \Gamma$ would contain a core curve of this annulus. Pick one such arc in each essential annular component. Let $\alpha$ be the union of these arcs. If $A_i$ is an essential annulus, then define $D_i$ to be $A_i \cut \alpha$. If $A_i$ is an inessential annulus, then let $D_i$ be the disc in $S$ containing $A_i$ that has boundary equal to a component of $\partial A_i$. If $A_i$ is a disc, let $D_i$ be $A_i$. Some of these discs may be nested, in which case discard the smaller disc. Thus, $D_1, \dots, D_m$ is a collection of discs as in Remark~\ref{Rem:StrengthenSlidingOffDiscs}. So, there is a sequence of edge swaps taking $\Gamma$ to a cellular spine that is disjoint from the interior of $D_1, \dots, D_m$. It must intersect the interior of each essential annulus $A_i$ in a single arc.

Unfortunately, the number of these edge swaps is bounded above by a linear function of the total length of the boundary of $D_1, \dots, D_m$, which depends not just on $\ell$ but also on the length of the arcs $\alpha$. To deal with this, we define a new cell structure $\calC'$ on $S$, as follows. Away from $A_1 \cup \dots \cup A_m$, this agrees with $\calC$, but within each $A_i$, the 2-cells are the components of $A_i \cut \Gamma$. Note that $\Gamma$ is still cellular with respect to $\calC'$. We can now bound the length of $\partial D_1, \dots, \partial D_m$ in $\calC'$ in terms of $\ell$. This length is at most the length of $\partial A_1, \dots, \partial A_m$ plus twice the length of $\alpha$. The length of $\alpha$ with respect to $\calC'$ is at most the number of vertices of $\Gamma$
plus the number of essential annular components of $A_1, \dots, A_m$. By an Euler characteristic argument, using the fact that each vertex has degree at least three, the number of vertices of $\Gamma$ is at most $4g(S)-2$. The number of essential annular components of $A_1, \dots, A_m$ is at most $m$. So, the length of $\partial D_1, \dots, \partial D_m$ in $\calC'$ is at most $\ell+8g(S)+2m$. Now apply \reflem{SlidingOffDiscs} and Remark~\ref{Rem:StrengthenSlidingOffDiscs} to turn $\Gamma$ into a spine that is cellular in $\calC'$ and that intersects the interior of $A_1 \cup \dots \cup A_m$ in the arcs $\alpha$. It is then cellular with respect to $\calC$.
\end{proof}

\begin{lemma}
\label{Lem:PachnerMoveSwap}
Let $\calT$ and $\calT'$ be triangulations of a closed surface $S$ that differ by a sequence of $n$ Pachner moves. Let $\Gamma$ be a subcomplex of $\calT$ that is a spine of $S$. Then there is a sequence of at most $n$ edge swaps and some isotopies taking $\Gamma$ to a spine that is a subcomplex of $\calT'$.
\end{lemma}

\begin{proof}
It suffices to consider the case where $n=1$, and so $\calT$ and $\calT'$ differ by a single Pachner move.
Some terminology: throughout this proof, edges will refer to edges in a spine, with endpoints on vertices of the spine of valence at least 3. We refer to edges of the triangulation $\calT$, which are not necessarily edges of $\Gamma$ even when they lie in $\Gamma$, by 1-cells. 

If the Pachner move is a 1-3 move, then there is nothing to prove as the 1-skeleton of $\calT$ is then a subcomplex of $\calT'$.

Suppose it is a 2-2 move, removing a 1-cell $e$ and inserting a new 1-cell $e'$. If $e$ is not part of $\Gamma$, then $\Gamma$ is a subcomplex of $\calT'$ and so no edge swaps are required. So suppose that $e$ is contained in $\Gamma$. Since $S \cut \Gamma$ is a disc, there is an arc $\alpha$ running from the midpoint of $e$ back to the midpoint of $e$ but on the other side of $e$ and that is otherwise disjoint from $\Gamma$. We may assume that $\alpha$ is disjoint from the vertices of $\calT$ and intersects each 1-cell of $\calT$ at most once. It must intersect at least one 1-cell $e''$ in the boundary of the square that is the union of the two triangles involved in the Pachner move. If both endpoints of $e''$ lie in $\Gamma$, let $e''' = e''$. Otherwise, let $e'''$ be the union of $e''$ and the third 1-cell of the triangle formed by $e''$ and $e$.
We perform the edge swap that adds $e'''$ to $\Gamma$ and removes the edge of $\Gamma$ containing $e$.

We now consider a 3-1 move. 
Let $e_1$, $e_2$ and $e_3$ be the three 1-cells of $\calT$ that are removed. If none of these are part of $\Gamma$, then we leave the spine unchanged. There cannot be just one of these 1-cells in $\Gamma$, since no vertex of $\Gamma$ has degree $1$. Suppose $\Gamma$ runs over exactly two 1-cells in $\{ e_1, e_2, e_3 \}$, say $e_1$ and $e_2$. These are two 1-cells of a triangle of $\calT$. The third 1-cell $e'$ of this triangle cannot lie in $\Gamma$, as $S \cut \Gamma$ is a single disc.
Hence, we may isotope $e_1 \cup e_2$ across the triangle to $e'$.
Suppose finally that all three of $e_1$, $e_2$ and $e_3$ are part of $\Gamma$. Let $e$ be the other 1-cell of the triangle formed by $e_1$ and $e_2$. Adding $e$ to $\Gamma$ and removing $e_1$ is an edge swap. We then remove $e_2$ and $e_3$ and add the third 1-cell of the triangle that they span. This is realised by an isotopy of the spine and so no edge swap is required.
\end{proof}

\begin{lemma}
\label{Lem:ToOneVertexTriangulation}
Let $\calT$ be a triangulation of a torus with $v$ vertices. Then there is a sequence of at most $4v$ Pachner moves taking $\calT$ to a 1-vertex triangulation.
\end{lemma}

\begin{proof} This is contained in the proof of \cite[Proposition~10.3]{Lackenby:Knottedness}, and so we only sketch the argument. Suppose $v > 1$, as otherwise we are done. The strategy is to apply at most $4$ Pachner moves to the triangulation, after which the number of vertices is reduced. 

If there is an edge of the triangulation with the same triangle on both sides, then one endpoint of the edge is a vertex with valence 1. It is possible to apply two 2-2 Pachner moves to increase this valence to $3$. Then one can apply a 3-1 Pachner move to remove this vertex.

So we may suppose that every edge of the triangulation has distinct triangles on both sides. Using the fact that the Euler characteristic of the torus is zero, there is a vertex with valence at most $6$. Suitably chosen 2-2 Pachner moves then reduce this to $3$. A 3-1 Pachner move can then be used to remove the vertex.
\end{proof}

\begin{lemma}
\label{Lem:BarycentricPachner}
Let $\calT$ be a triangulation of a compact surface $S$ with $t$ triangles. Then the barycentric subdivision $\calT^{(1)}$ is obtained from $\calT$ by $4t$ Pachner moves and an isotopy.
\end{lemma}

\begin{proof}
First perform a 1-3 Pachner move to each triangle of $\calT$. Each original edge of $\calT$ is then adjacent to two new triangles. Choose one of the two, and perform a 1-3 Pachner move in that triangle. Then perform the 2-2 move that removes the edge. The resulting triangulation is isotopic to $\calT^{(1)}$. In total, we have performed $4t$ Pachner moves.
\end{proof}

\begin{lemma}\label{Lem:SpineAndShortCurve}
Let $T$ be a torus equipped with a cell structure. Let $\Gamma$ be a cellular spine for $T$. Let $C$ be a cellular essential simple closed curve with length $\ell$. Then there exists a spine for $T$ that is obtained from $\Gamma$ by at most $24 + 4\ell$ edge swaps and that contains $C$.
\end{lemma}

\begin{proof} The annulus $T \cut C$ has boundary length $2 \ell$. Apply \reflem{SlidingOffAnnuli} to turn $\Gamma$ into a spine $\Gamma'$ that intersects the interior of $T \cut C$ in a single arc, using at most $24+4\ell$ edge swaps. The spine $\Gamma'$ must contain all of $C$, as otherwise $S \cut \Gamma'$ would contain an essential simple closed curve.
\end{proof}

%%%%%%%%%%%%%%%%%%%%%%%%%%%%%%%%%%%%%%%%%%%%%%%%%%%%%%%%%%%%%%%%%
\section{Triangulations of a torus}
\label{Sec:TriangulationsTorus}

In this section, we recall a description of the space $\mathrm{Tr}(T^2)$ of all 1-vertex triangulations of a torus.

The \emph{Farey graph} is a graph with vertex set $\mathbb{Q} \cup \{ \infty \}$, and where two vertices $p/q$ and $r/s$ are joined by an edge if and only if $|ps - qr| = 1$. Here, we assume that the fractions are in their lowest terms and that $\infty = 1/0$. Now, $\mathbb{Q} \cup \{ \infty \}$ is a subset of $\mathbb{R} \cup \{ \infty \}$, which is the circle at infinity of the upper-half plane. We can realise each edge of the Farey graph as an infinite geodesic in the hyperbolic plane; see Figure~\ref{Fig:Farey}. The edges of the Farey graph form the edges in a tessellation of $\mathbb{H}^2$ by ideal triangles. We call this the \emph{Farey tesselation}. Each triangle has three points on the circle at infinity, and these correspond to three slopes on the torus, with the property that any two of these slopes intersect once. Given three such slopes, we can realise them as Euclidean geodesics in the torus, which we think of as $\mathbb{R}^2  / \mathbb{Z}^2$. We can arrange that these geodesics each go through the image of the origin, and hence all intersect at this point. Thus, this forms a 1-vertex triangulation of the torus. Conversely, given any 1-vertex triangulation of the torus, we may isotope the vertex to the origin, and then isotope each of the edges to Euclidean geodesics. Thus, we see that there is a 1-1 correspondence between 1-vertex triangulations of the torus, up to isotopy, and ideal triangles in the Farey tessellation. 

\begin{figure}
  \includegraphics[width=3.5in]{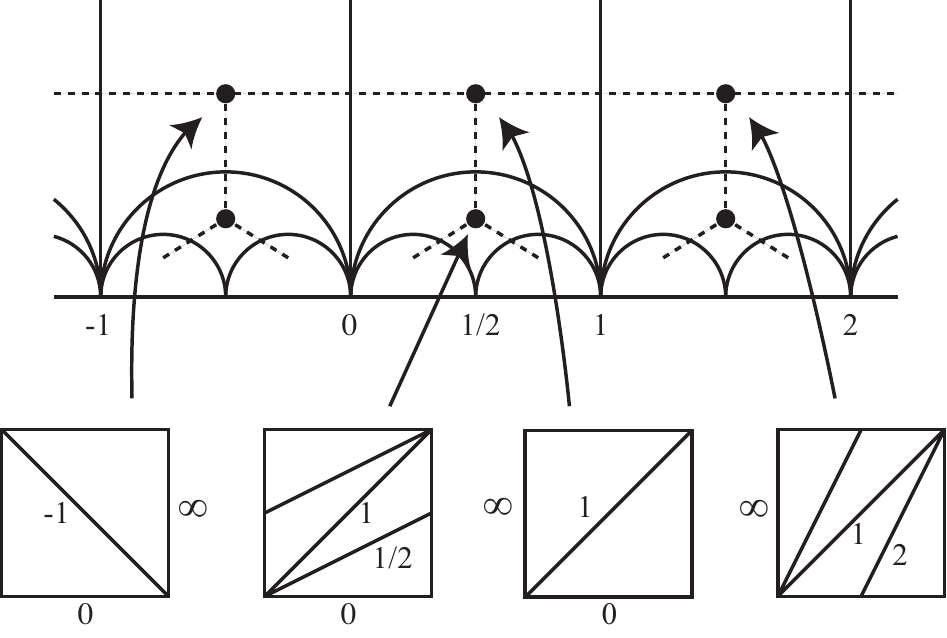}
  \caption{The Farey graph and the dual Farey tree}
  \label{Fig:Farey}
\end{figure}

When a 2-2 Pachner move is performed, this removes one of the edges of the triangulation, forming a square, and then inserts the other diagonal of the square. The remaining two edges of the triangulation are preserved, and these correspond to an edge of the Farey graph. Thus we see that two triangulations differ by a 2-2 Pachner move if and only if their corresponding ideal triangles in the Farey tessellation share an edge of the Farey graph.

It is natural to form the dual of the Farey tessellation, which is the \emph{Farey tree}. This has a vertex for each ideal triangle of the Farey tessellation, and two vertices of the Farey tree are joined by an edge if and only if the dual triangles share an edge. As the name suggests, this is a tree. The above discussion has the following immediate consequence.

\begin{theorem}\label{Thm:TrIsFarey}
The graph $\mathrm{Tr}(T^2)$ is isomorphic to the Farey tree. \qed
\end{theorem}

\subsection{A Cayley graph for $\mathrm{PSL}(2, \mathbb{Z})$}
\label{SubSec:PLS2Z}

The mapping class group of the torus is isomorphic to $\mathrm{SL}(2, \mathbb{Z})$. 
It was shown by Serre \cite{Serre:Trees} that $\mathrm{SL}(2, \mathbb{Z})$ is isomorphic to the amalgamated free product of $\mathbb{Z}_4$ and $\mathbb{Z}_6$, amalgamated over the subgroups of order 2. The non-trivial element in the amalgamating subgroup is the matrix $-I$, which is central. If we quotient by this subgroup, the result is $\mathrm{PSL}(2, \mathbb{Z})$, which is isomorphic to $\mathbb{Z}_2 \ast \mathbb{Z}_3$. The factors are generated by 
\[
S = 
\left(
\begin{matrix}
0 & -1 \\
1 & 0 \\
\end{matrix}
\right)
\qquad
T = 
\left(
\begin{matrix}
0 & -1 \\
1 & -1 \\
\end{matrix}
\right).
\]
The Farey tree is closely related to the Cayley graph for $\mathrm{PSL}(2, \mathbb{Z})$ with respect to these generators. This group acts on upper half space by isometries. It preserves $\mathbb{Q} \cup \{ \infty \}$ in the circle at infinity, and hence it preserves the Farey tesselation and the dual Farey tree.
The Cayley graph for $\mathrm{PSL}(2, \mathbb{Z})$ with respect to these generators embeds in upper half space as follows. We set the vertex $v$ corresponding to the identity element of $\mathrm{PSL}(2, \mathbb{Z})$ to lie at $\epsilon + i$ for some small real $\epsilon > 0$. The images of this point $v$ under the action of $\mathrm{PSL}(2, \mathbb{Z})$ form the vertices of the Cayley graph. Emanating from the vertex $v$ there are oriented edges, joining $v$ to $Sv$ and $Tv$. The images of these edges under the action of $\mathrm{PSL}(2, \mathbb{Z})$ form the edges of the Cayley graph. The graph is shown in Figure \ref{Fig:CayleyGraph}. Note that $v$ lies in the triangle with corners $0$, $1$ and $\infty$. The stabiliser of this triangle in $\mathrm{PSL}(2, \mathbb{Z})$ is $\{ T^i: i = 0,1,2 \}$. Hence, there are three vertices of the Cayley graph in this triangle that are connected by edges labelled by $T$. The edge joining $v$ to $Sv$ intersects the geodesic joining $0$ and $\infty$ and is disjoint from the remaining edges of the Farey graph. Hence, each $S$-labelled edge of the Cayley graph is dual to an edge of the Farey graph. Each edge of the Farey graph is associated with two such $S$-labelled edges, which join the same pair of the vertices.

Thus, in summary, the Cayley graph is obtained from the Farey tree as follows. Replace each vertex of the Farey tree by a little triangle, with edges labelled by $T$. Replace each edge of the Farey tree by two edges of the Cayley graph labelled by $S$.

\begin{figure}
  \includegraphics[width=3.5in]{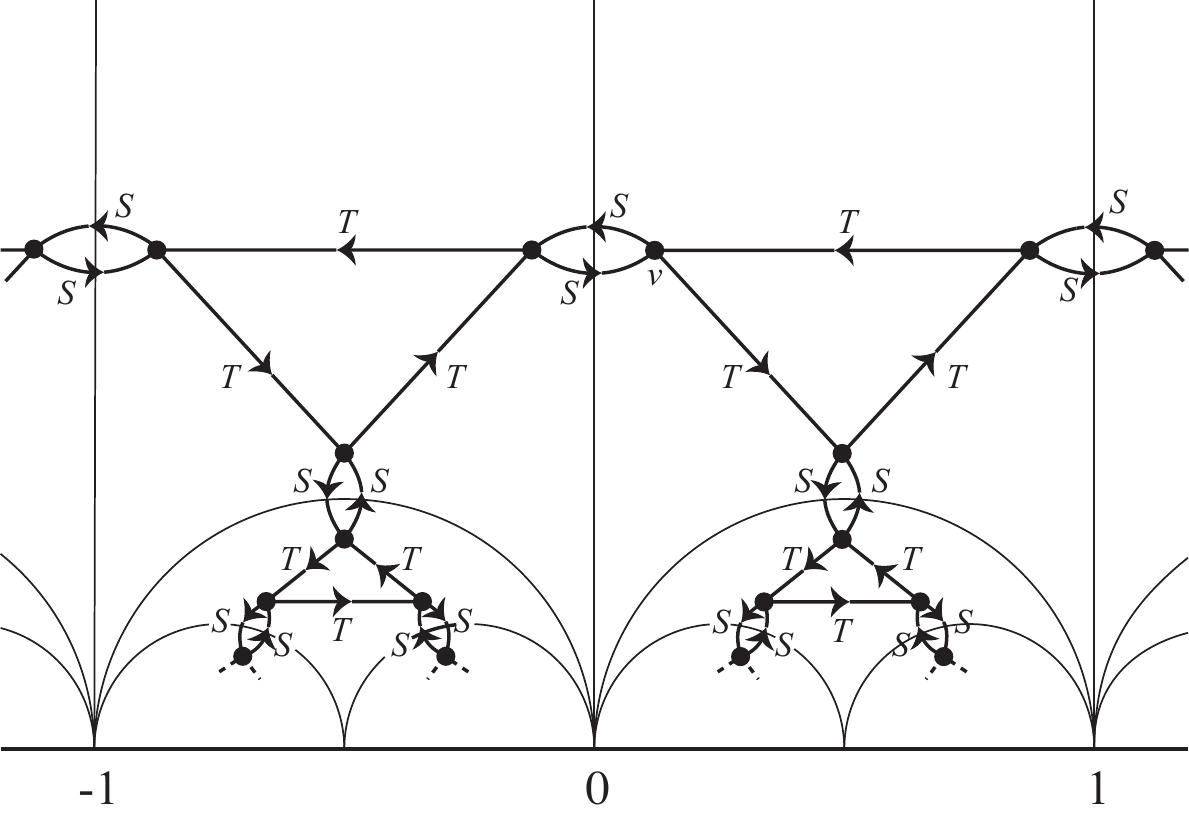}
  \caption{The Cayley graph of $\mathrm{PSL}(2, \mathbb{Z})$ with respect to the generators $S$ and $T$}
  \label{Fig:CayleyGraph}
\end{figure}

\subsection{Continued fractions}

There is a well-known connection between continued fractions and the Farey tessellation.

A \emph{continued fraction} for a rational number $r$ is an expression
\[ r = a_0+\cfrac{1}{a_1+\cfrac{1}{a_2+\cfrac{1}{\cdots+\cfrac{1}{a_n}}}} \]
where each $a_i$ is an integer.
This is written $r = [a_0, a_1, \dots, a_n]$.

One can also consider an irrational number $r$, which also has a \emph{continued fraction expansion} $[a_0, a_1, \dots ]$. This means that if $r_n = [a_0, a_1, \dots, a_n]$, then $r_n \rightarrow r$ as $n \rightarrow \infty$. We will focus on the case where $a_i$ is positive for each $i>0$. Subject to this condition, every real number $r$ has a unique continued fraction expansion, which we shall call \emph{the} continued fraction expansion for $r$.

The continued fraction expansion of $r$ is \emph{periodic} if there is a non-negative integer $k$ and an even positive integer $t$ such that $a_{i+t} = a_i$ for every $i \geq k$. The smallest such $t$ is the \emph{length} of the periodic part. The following is a well-known result of Lagrange; see, for example \cite{Davenport}.

\begin{lemma}[Lagrange]\label{Lem:Lagrange}
  The continued fraction expansion of a real number $r$ is periodic exactly when $\mathbb{Q}(r)$ is a quadratic extension of $\mathbb{Q}$.
  This happens exactly when $r = p + q \sqrt{d}$ for some square free integer $d>1$ and some rational numbers $p$ and $q$ where $q \not=0$. Moreover, for fixed $d$, two real numbers $p + q \sqrt{d}$ and $p' + q' \sqrt{d}$, for $p, p' \in \mathbb{Q}$ and $q,q' \in \mathbb{Q} \setminus \{ 0 \}$ have the same periodic part. That is, if $[a_0, a_1, a_2, \dots]$ and $[a'_0, a'_1, a'_2, \dots]$ are their continued fraction expansions, then there are integers $l$ and $m$ such that $a'_i = a_{i+m}$ for all $i \geq l$.
\end{lemma}

One can read off the continued fraction expansion of a positive real number $r$ from the Farey tessellation, as follows. Consider any hyperbolic geodesic $\gamma$ starting in the hyperbolic plane on the imaginary axis, and ending at $r$ on the circle at infinity.  It intersects each triangle of the Farey tessellation in at most one arc; see \reffig{ContinuedFraction}. As one travels along $\gamma$ and one enters such a triangle, it either goes to the right or the left in this triangle, except when $r$ is rational and $\gamma$ lands on $r$. So, when $r$ is irrational, one reads off the \emph{cutting sequence} of $\gamma$, which is a sequence of lefts and rights, written as $L^{a_0} R^{a_1} L^{a_2} \dots$. Then the continued fraction expansion of $r$ is $[a_0, a_1, a_2, \dots]$. When $r$ is rational, we also get a cutting sequence, but we must be careful about the final triangle that $\gamma$ runs through. Here, $\gamma$ goes neither left nor right, but instead straight on towards $r$. We view this final triangle as giving a final $L$ or $R$ to the cutting sequence, where $L$ or $R$ is chosen to be the same as the previous letter. Thus, we obtain a sequence $L^{a_0} R^{a_1} L^{a_2} \dots L^{a_n}$ or $L^{a_0} R^{a_1} L^{a_2} \dots R^{a_n}$ with $a_n \geq 2$. Then $[a_0, a_1, a_2, \dots, a_n]$ is the continued fraction expansion of $r$.

\begin{figure}
  \includegraphics[width=3.5in]{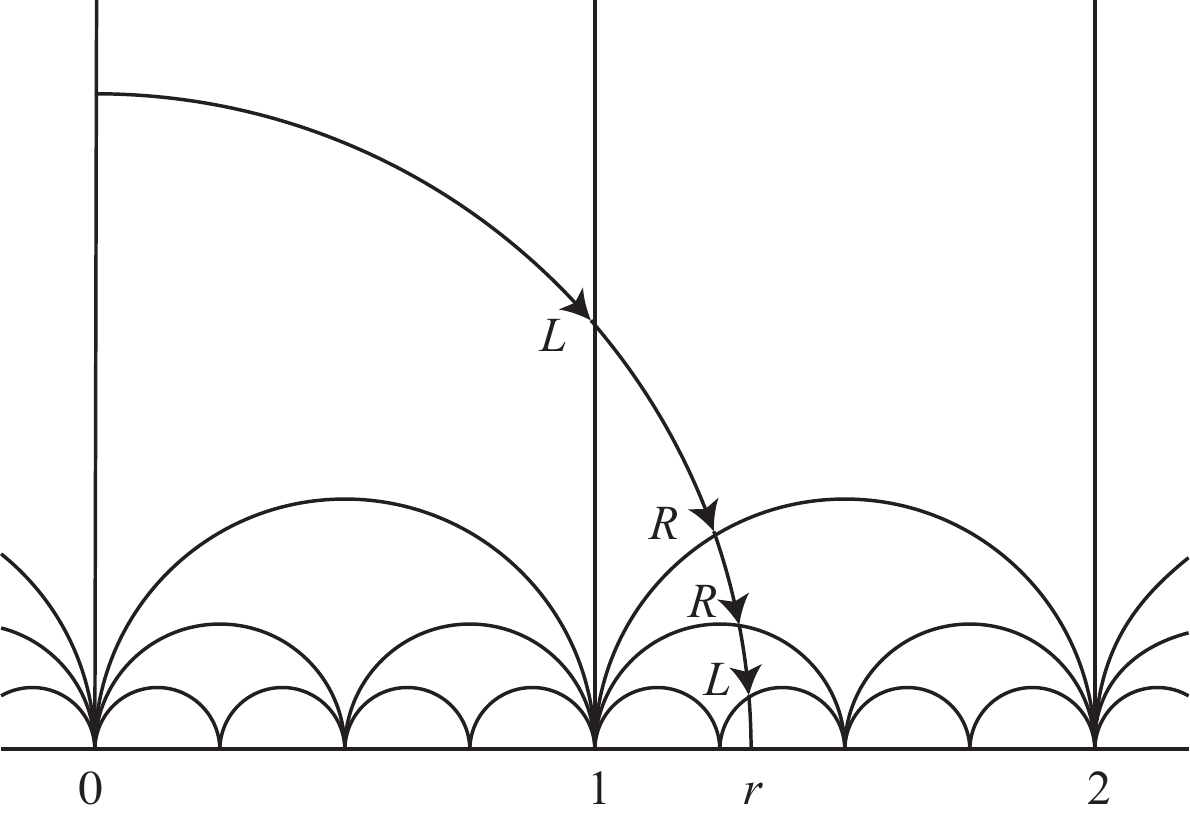}
  \caption{The cutting sequence determined by a geodesic starting on the imaginary axis and ending at $r$}
  \label{Fig:ContinuedFraction}
\end{figure}

The following lines in the Farey tree will play an important role in our analysis of lens spaces.

\begin{definition}\label{Def:FareyLine}
For $p/q \in \mathbb{Q} \cup \{ \infty \}$, the line $L(p/q)$ in the Farey tree is the union of edges that are dual to an edge of the Farey graph emanating from $p/q$.
\end{definition}

\begin{figure}
  \includegraphics[width=3.5in]{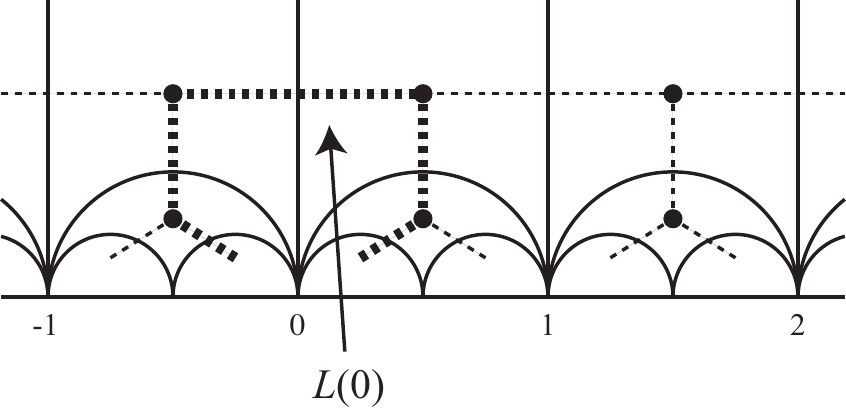}
  \caption{The line $L(0)$ in the Farey tree}
  \label{Fig:FareyHorocycle}
\end{figure}

\begin{lemma}
\label{Lem:LinesInFareyTree}
When $0 < q < p$, the distance in the Farey tree between the lines $L(q/p)$ and $L(\infty)$ is $(\sum_{i=0}^n a_i) - 1$ where $[a_0, \dots, a_n]$ is the continued fraction of expansion of $p/q$ with each $a_i$ positive.
\end{lemma}

\begin{proof} The line $L(\infty)$ runs parallel to the horocycle $\{ (x,y): y = 1 \}$ in upper half-space. The line $L(q/p)$ forms a loop starting and ending at $q/p$. Let $\gamma$ be the vertical geodesic in the half plane running from $\infty$ to the point $q/p$. As it comes from infinity, it hits $L(\infty)$, then it intersects various edges in the Farey graph, and then it hits $L(q/p)$. This determines a path in the Farey tree from $L(\infty)$ to $L(q/p)$. There is no shorter path, because each edge $e$ of the Farey tessellation crossed by $\gamma$ separates $L(\infty)$ from $L(q/p)$, and so any path in the Farey tree from $L(\infty)$ to $L(q/p)$ must run along the edge dual to $e$.

A closely related geodesic $\gamma'$ determines the continued fraction expansion for $q/p$. This starts on the imaginary axis and ends at $q/p$. But because $0 < q/p < 1$, $\gamma$ and $\gamma'$ hit the same edges of the Farey graph (except the edge that forms the imaginary axis). Note that the continued fraction expansion of $q/p$ is $[0,a_0, \dots, a_n]$. Hence, $(\sum_{i=0}^n a_i) - 1$ is exactly the length of the path in the Farey tree joining $L(\infty)$ to $L(q/p)$. 
\end{proof}

%%%%%%%%%%%%%%%%%%%%%%%%%%%%%%%%%%%%%%%%%%%%%%%%%%%%%%%%%%%%%%%%%

\section{Train track splitting sequences}
\label{Sec:TrainTrack}

We will estimate distances in triangulation graphs using train tracks. We start by recalling some terminology.

A \emph{pre-track} is a graph $\tau$ smoothly embedded in the interior of a surface $S$ such that at each vertex $v$, the following hold:
\begin{enumerate}
\item there are three edges coming into $v$;
\item these edges all have non-zero derivative at $v$, all of which lie in the same tangent line;
\item one edge approaches $v$ along this line from one direction, and the other two edges approach from the other direction.
\end{enumerate}
The vertices of $\tau$ are called \emph{switches} and the edges are called \emph{branches}.

Each component $R$ of $S \cut \tau$ is a surface, but its boundary is not necessarily smooth. Its boundary is composed of a union of arcs, one for each edge of $\tau$. When two of these arcs cannot be combined into a single smooth arc, their point of intersection is a \emph{cusp}. The \emph{index} of $R$ is equal to $\chi(R)$ minus half the number of cusps of $\partial R$. We say that $\tau$ is a \emph{train track} if each component of $S \cut \tau$ has negative index. If we add up the index of the components of $S \cut \tau$, the result is $\chi(S)$. Hence, we deduce that the number of complementary regions of a train track is at most $-2 \chi(S)$.

A train track is \emph{filling} if each component of $S \cut \tau$ is either a disc or an annular neighbourhood of a component of $\partial S$. When $\tau$ is filling, it is dual to a triangulation of $S$, possibly with some ideal vertices.

Two train tracks differ by a \emph{split} or a \emph{slide} if one is obtained from the other by one of the modifications shown in Figure \ref{Fig:SplitSlide}.

\begin{figure}
%% Accompanies image file 'SplitSliderTrainTrack.pdf' (pdf, eps, ps)
\begingroup%
  \makeatletter%
  \providecommand\color[2][]{%
    \errmessage{(Inkscape) Color is used for the text in Inkscape, but the package 'color.sty' is not loaded}%
    \renewcommand\color[2][]{}%
  }%
  \providecommand\transparent[1]{%
    \errmessage{(Inkscape) Transparency is used (non-zero) for the text in Inkscape, but the package 'transparent.sty' is not loaded}%
    \renewcommand\transparent[1]{}%
  }%
  \providecommand\rotatebox[2]{#2}%
  \newcommand*\fsize{\dimexpr\f@size pt\relax}%
  \newcommand*\lineheight[1]{\fontsize{\fsize}{#1\fsize}\selectfont}%
  \ifx\svgwidth\undefined%
    \setlength{\unitlength}{210.53871346bp}%
    \ifx\svgscale\undefined%
      \relax%
    \else%
      \setlength{\unitlength}{\unitlength * \real{\svgscale}}%
    \fi%
  \else%
    \setlength{\unitlength}{\svgwidth}%
  \fi%
  \global\let\svgwidth\undefined%
  \global\let\svgscale\undefined%
  \makeatother%
  \begin{picture}(1,0.72668617)%
    \lineheight{1}%
    \setlength\tabcolsep{0pt}%
    \put(0,0){\includegraphics[width=\unitlength,page=1]{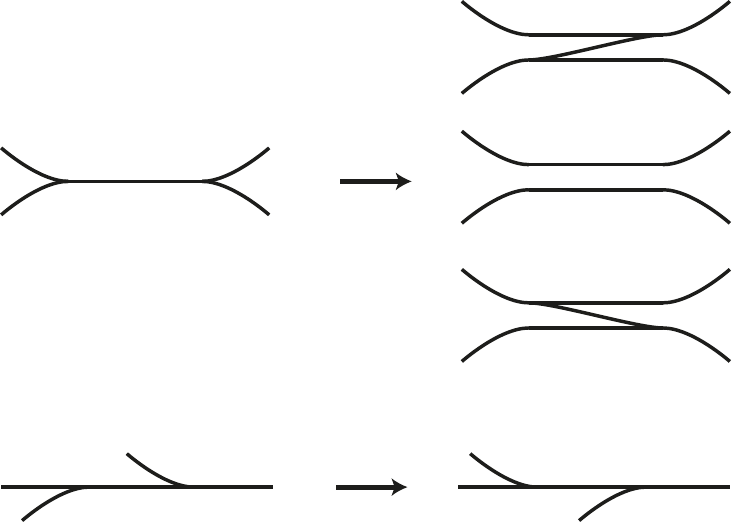}}%
    \put(0.46284952,0.41901961){\color[rgb]{0,0,0}\makebox(0,0)[lt]{\lineheight{1.25}\smash{\begin{tabular}[t]{l}split\end{tabular}}}}%
    \put(0.4554366,0.00071895){\color[rgb]{0,0,0}\makebox(0,0)[lt]{\lineheight{1.25}\smash{\begin{tabular}[t]{l}slide\end{tabular}}}}%
  \end{picture}%
\endgroup%

  \caption{Splits and slide applied to a train track}
  \label{Fig:SplitSlide}
\end{figure}

Pick a point in the interior of each branch. Cutting the branch at such a point creates two intervals, which are called \emph{half-branches}. A half-branch is called \emph{small} if at the cusp at its endpoint, there is another half-branch coming in from the same direction. If a half-branch is not small, it is \emph{large}.

For any train track $\tau$, its regular neighbourhood $N(\tau)$ is naturally a union of intervals called \emph{fibres}, and there is a collapsing map $N(\tau) \rightarrow \tau$ that collapses each fibre to a point. A curve $C$ is said to \emph{carried} by $\tau$ if $C$ is embedded in $N(\tau)$ and is transverse to all the fibres.

A train track $\tau$ is \emph{transversely recurrent} if for each branch of $\tau$, there is a simple closed curve $C$ transverse to $\tau$ that intersects this branch and such that $\tau \cup C$ does not have a complementary region that is a bigon. By a \emph{bigon}, we mean a component of $S \cut (C \cup \tau)$ that is a disc with boundary consisting of the union of two arcs, one lying in $C$, the other lying in $\tau$ and having no cusps. The train track $\tau$ is \emph{recurrent} if for each branch of $\tau$, there is a simple closed curve carried by $\tau$ that runs over the branch. It is \emph{birecurrent} if it is both recurrent and transversely recurrent. We say that \emph{the set of curves carried by $\tau$ fills $S$} if, for every essential simple closed curve $C$ in $S$, there is a curve carried by $\tau$ that cannot be isotoped off $C$. The train track is then said to be \emph{filling}.

The following is essentially due to Masur, Mosher and Schleimer \cite{MasurMosherSchleimer}.

\begin{theorem} 
\label{Thm:QuasiGeodesicSplittingSequence}
Let $\tau$ and $\tau'$ be filling birecurrent train tracks in a closed orientable surface $S$ of genus at least 2.
Suppose that there is a sequence of splits and slides taking $\tau$ to $\tau'$. Let $\mathcal{T}$ and $\mathcal{T}'$ be the triangulations dual to $\tau$ and $\tau'$. Let $n = -2\chi(S)$. Then, the distance in $\mathrm{Tr}(S;n)$ between $\mathcal{T}$ and $\mathcal{T}'$ is, up to a bounded multiplicative error, equal to the number of splits. This bound only depends on the Euler characteristic of $S$.
\end{theorem}

\begin{proof}
In \cite[Section 6.1]{MasurMosherSchleimer}, the \emph{marking graph} $\mathcal{M}(S)$ is defined. This is quasi-isometric to the mapping class group of $S$. In \cite[Section 6.1]{MasurMosherSchleimer}, a map from filling birecurrent train tracks to $\mathcal{M}(S)$ is defined. Composing this with the quasi-isometries $\mathcal{M}(S) \rightarrow \mathrm{MCG}(S) \rightarrow \mathrm{Tr}(S;n)$ we obtain a map from filling birecurrent train tracks to $\mathrm{Tr}(S;n)$. This is a bounded distance from the map that sends each filling birecurrent train track to its dual triangulation.

We are supposing that there is a sequence of splits and slides taking $\tau$ to $\tau'$. Then by \cite[Theorem~6.1]{MasurMosherSchleimer},
a sequence of such splits and slides is sent to a quasi-geodesic in $\mathcal{M}(S)$, with quasi-geodesic constants depending only on $S$. The length of this quasi-geodesic is the number of splits in the sequence. We compose this with the quasi-isometries $\mathcal{M}(S) \rightarrow \mathrm{MCG}(S) \rightarrow \mathrm{Tr}(S;n)$, and we obtain a quasi-geodesic in $\mathrm{Tr}(S;n)$. Its start and end vertices are a bounded distance from $\mathcal{T}$ and $\mathcal{T}'$, the bound depending only on $S$. Hence, the distance in $\mathrm{Tr}(S;n)$ between $\mathcal{T}$ and $\mathcal{T}'$ is, up to a bounded multiplicative error, equal to the number of splits.
\end{proof}

Now suppose that the train track $\tau$ has a transverse measure $\mu$. (We refer to \cite{FLP} for the definition of transverse measures and their relationship with pseudo-Anosov homeomorphisms.) At any large branch of $\tau$, one may split $\tau$ in three possible ways, but only one of these ways is compatible with $\mu$. The result is a measured train track $(\tau', \mu')$. 

A \emph{maximal split} on a measured train track $(\tau, \mu)$ is obtained by performing a measured split at each large branch of $\tau$. 
The following was proved by Agol~\cite[Theorem~3.5]{Agol:IdealTriangulations}.

\begin{theorem}
\label{Thm:Agol}
Let $S$ be a compact orientable surface.
Let $\phi \colon S \rightarrow S$ be a pseudo-Anosov homeomorphism, and let $(\tau, \mu)$ be a measured train track that carries its stable measured lamination. Let $\lambda$ be its dilatation. For each positive integer $i$, let $(\tau_i, \mu_i)$ be the result of performing a sequence of $i$ maximal splits to $(\tau, \mu)$. Then there are integers $n>0$ and $m \geq 0$ such that, for each $i \geq m$, $\tau_{n+i} = \phi(\tau_i)$ and $\mu_{n+i} = \lambda^{-1} \mu_i$.
\end{theorem}

%%%%%%%%%%%%%%%%%%%%%%%%%%%%%%%%%%%%%%%%%%%%%%%%%%%%%%%%%%%%%%%%%
\section{Homeomorphisms of the torus}
\label{Sec:HomeoTorus}

The famous classification of orientation-preserving homeomorphisms of closed orientable surfaces into periodic, reducible and pseudo-Anosov \cite{Thurston:DiffeomorphismsSurfaces} is a generalisation of the special case of the torus. In this case, the third category is known as \emph{linear Anosov}, which we can define to be isotopic to a linear map with determinant $ 1$, and with irrational real eigenvalues.

An orientation-preserving homeomorphism of the torus induces an action on the homology of the torus and hence gives an element of $\mathrm{SL}(2, \mathbb{Z})$. There is a homomorphism $\mathrm{SL}(2, \mathbb{Z})$ to the isometry group of the hyperbolic plane. Thus, our homeomorphism of the torus induces an isometry of the hyperbolic plane that preserves the Farey tessellation, and hence is an isometry of the Farey tree. An alternative way of viewing this action is to note that the Farey tree is $\mathrm{Tr}(T^2)$ and any homeomorphism of the torus naturally induces an isometry of $\mathrm{Tr}(T^2)$.

Recall that any isometry of a tree either has a fixed point or has an \emph{invariant axis}. This is a subset of the tree isometric to the real line, such that the isometry acts as non-trivial translation upon this line. Any subset of the Farey tree isometric to the real line has two well-defined endpoints on the circle at infinity, although these endpoints need not be distinct. We can now see the classification of orientation-preserving homeomorphisms of the torus in terms of the action on the Farey tree:
\begin{enumerate}
\item A homeomorphism is periodic if and only if its action on the Farey tree has  a fixed point in the interior.
\item A homeomorphism is reducible and not periodic if and only if its action on the Farey tree has an invariant axis, but the endpoints of this axis are the same point on the circle at infinity.
\item A homeomorphism is linear Anosov if and only if its action on the Farey tree has an invariant axis, and the endpoints of the axis are distinct points on the circle at infinity.
\end{enumerate}

The following is well-known (see for example \cite[Section 0]{CassonBleiler}).

\begin{lemma}\label{Lem:AnosovTrace}
A matrix $A\in \SL(2,\ZZ)$ acts on the torus as a linear Anosov homeomorphism if and only if its trace $\mathrm{tr}(A)$ satisfies $|\tr(A)|>2$.
\end{lemma}

\begin{proof}
The matrix $A$ projects to an element of $\PSL(2,\ZZ)$, which is a subgroup of orientation-preserving isometries of the hyperbolic plane. This isometry induces the action of $A$ on the Farey tessellation and hence on the Farey tree.

Consider first the case that the action has a fixed point in the interior. If $A$ has rows $(a,b)$ and $(c,d)$, a fixed point is an element $x$ with positive imaginary part such that $(ax+b)/(cx+d)=x$. Solving for $x$, this gives a quadratic polynomial with discriminant $\sqrt{\tr(A)^2-4}$ and with highest order term $cx^2$. Thus there is a fixed point with positive imaginary part if and only if $|\tr(A)|<2$ and $c \not=0$. 
Note however that if $c =0$, then the condition that $A$ has determinant $1$ forces $\tr(A)$ to be equal to $\pm 2$. Thus there is a fixed point with positive imaginary part if and only if $|\tr(A)|<2$.
 
Consider next the eigenvectors of $A$. Specifically, $(u, 1)^T$ is an eigenvector of $A$ if and only if $u/1$ is an endpoint of an invariant axis of the action of $A$ on the Farey tree. Now, the characteristic polynomial for $A$ is $x^2 - \mathrm{tr}(A) x + 1$, with roots
$\left(\tr(A) \pm \sqrt{\tr(A)^2 - 4}\right)/2$.
Thus $A$ has distinct real eigenvalues if and only if $|\tr(A)|>2$; this is the case $A$ induces a linear Anosov. The remaining case, $|\tr(A)|=2$, corresponds to the case $A$ is reducible and not periodic.
\end{proof}

When an isometry $\phi$ of a tree has an invariant axis, then this is also the invariant axis for any non-zero power of $\phi$. Hence, we have the following result.

\begin{lemma} \label{Lem:StableTransTorus}
Let $A$ be a homeomorphism of the torus. Then for any integer $n$, $\ell_{\mathrm{Tr}(T^2)}(A^n) = n \, \ell_{\mathrm{Tr}(T^2)}(A)$. Hence, the stable translation length satisfies $\overline{\ell}_{\mathrm{Tr}(T^2)}(A) = \ell_{\mathrm{Tr}(T^2)}(A)$.
\end{lemma}

\subsection{Translation length of a linear Anosov}

The following well known proposition gives the length of the translation in the Farey tree of an Anosov in terms of continued fractions. 

\begin{proposition}\label{Prop:AnosovDistance}
Let $A \in \mathrm{SL}(2, \mathbb{Z})$ act as a linear Anosov homeomorphism on the torus. Let $\overline{A}$ be the image of $A$ in $\mathrm{PSL}(2, \mathbb{Z})$. Suppose that $\overline{A}$ is $B^n$ for some positive integer $n$ and some matrix $B \in  \mathrm{PSL}(2, \mathbb{Z})$ that is not a proper power.  Let $(a_r, \dots, a_s)$ denote the periodic part of the continued fraction expansion of $\sqrt{\mathrm{tr}(A)^2 - 4}$. Then the translation distance of $A$ in the Farey tree is $n \sum_{i=r}^s a_i$.
\end{proposition}

\begin{proof}
As in the proof of \reflem{AnosovTrace}, the matrix $A$ corresponds to a linear Anosov homeomorphism when $|\tr(A)|>2$, with eigenvalues
\[ \frac{\mathrm{tr}(A) \pm \sqrt{\mathrm{tr}(A)^2 - 4}}{2}. \]
Let $\lambda$ be either of these eigenvalues. Then the determinant of the matrix $A - \lambda I$ is zero, and hence the two rows are multiples of each other. Let $(a, b)$ be one of its rows. Suppose $(u,1)^T$ is an eigenvector for $A$, so $u/1$ is an endpoint of the invariant axis $\gamma$ for $A$. Then $au + b =0$ and so $u = - b/a$. Thus, we deduce that $u$ lies in $\mathbb{Q}(\lambda) = \mathbb{Q}(\sqrt{\mathrm{tr}(A)^2 - 4})$. So, the periodic part of the continued fraction of $u$ is equal to the periodic part of the continued fraction expansion of $\sqrt{\mathrm{tr}(A)^2 - 4}$ by \reflem{Lagrange}. 

Let $\gamma'$ be a geodesic starting at a point in the hyperbolic plane on the imaginary axis and ending at $u$ on the circle at infinity. The edges of the Farey graph that it crosses determines the cutting sequence $L^{a_0} R^{a_1} L^{a_2} \dots$ for $\gamma'$ and hence the continued fraction expansion $[a_0, a_1, \dots]$ for $u$. This cutting sequence is eventually the same as that of the invariant axis $\gamma$. In particular, they have the same periodic parts. Now, as $\gamma$ is the axis of $A$, its cutting sequence is periodic. However, the length of the corresponding path in the Farey tree may be a multiple of this period. This happens precisely when $\overline{A}$ is $B^n$ for some integer $n$ and some matrix $B \in  \mathrm{PSL}(2, \mathbb{Z})$ that is not a proper power. Hence, translation distance of $A$ in the Farey tree is $n \sum_{i=r}^s a_i$.
\end{proof}

\subsection{Moving between vertices in the Farey tree}

\begin{lemma} \label{Lem:ConstructTorusHomeo}
Given any two ideal triangles $a_1b_1c_1$ and $a_2b_2c_2$ of the Farey tessellation, there is a homeomorphism $\phi$ of the torus such that $\phi(a_1) = a_2$, $\phi(b_1) = b_2$, $\phi(c_1) = c_2$. This is unique up to isotopy and composition by the map $-\mathrm{id}$. It is orientation-preserving if and only if $\phi$ preserves the cyclic ordering of the vertices around the circle at infinity.
\end{lemma}

\begin{proof}
Since the slopes $a_1$ and $b_1$ have intersection number $1$, we may choose a basis for the first homology of the torus so that $a_1 = (1,0)$ and $b_1 = (0,1)$, when these slopes are oriented in some way. The function $a_1 \mapsto a_2$ and $b_1 \mapsto b_2$ may be realised by an element of $\mathrm{GL}(2, \mathbb{Z})$. Note that the determinant of this map is indeed $\pm 1$, since $a_2$ and $b_2$ are Farey neighbours. This linear map sends $c_1$ to a slope that has intersection number one with both $a_2$ and $b_2$. If this slope is not $c_2$, then pre-compose the linear map by $(1,0) \mapsto (-1,0)$ and $(0,1) \mapsto (0,1)$. This gives the required homeomorphism $\phi$.

To establish uniqueness, it suffices to check that if $\phi(a_1) = a_1$, $\phi(b_1) = b_1$ and $\phi(c_1) = c_1$ then $\phi$ is isotopic to $\pm \mathrm{id}$. But if the linear map $\phi$ sends $(1,0)$ to $\pm(1,0)$, sends $(0,1)$ to $\pm (0,1)$ and sends $(1,1)$ to $\pm (1,1)$, then $\phi$ is $\pm \mathrm{id}$.

We now show that $\phi$ is orientation-preserving if and only if $\phi$ preserves the cyclic ordering of the vertices around the circle at infinity. We established above that $\phi$ is either an element of $\mathrm{SL}(2, \mathbb{Z})$ or a composition of an element of $\mathrm{SL}(2, \mathbb{Z})$ with a reflection. In the former case, $\phi$ is orientation-preserving and realised by a M\"obius transformation of upper half-space, which therefore preserves the cyclic ordering of triples in the circle at infinity. In the latter case, $\phi$ is orientation-reversing and reverses the cyclic ordering of the vertices.
\end{proof}

\begin{lemma}\label{Lem:Anosov}
Suppose $a_1b_1c_1$ and $a_2b_2c_2$ are the vertices of distinct ideal triangles of the Farey tessellation. Let $\alpha$ be the unique embedded path in the Farey tree joining the centre of $a_1b_1c_1$ to the centre of $a_2b_2c_2$. Let $e$ be the first edge of the Farey graph that $\alpha$ crosses. Let $\phi$ be as in \reflem{ConstructTorusHomeo}. Then $\phi$ acts on the Farey tree, and so sends $\alpha$ to an arc $\phi(\alpha)$. Suppose that $\phi(\alpha)$ leaves the triangle $a_2b_2c_2$ by a different edge from the one $\alpha$ came in through. Suppose also that $\phi(e)$ and $e$ do not share a vertex. Then $\phi$ is linear Anosov. Moreover, the axis of $\phi$ in the Farey tree runs through the vertices dual to $a_1b_1c_1$ and $a_2b_2c_2$. 
\end{lemma}

\begin{proof} The infinite line $\bigcup_{n = -\infty}^\infty \phi^n(\alpha)$ forms an invariant axis. The endpoints of this axis on the circle at infinity are distinct, because they are separated by the endpoints of $e$ and $\phi(e)$. Hence, as discussed above, $\phi$ is linear Anosov.
\end{proof}

\begin{proposition} \label{Prop:PseudoAnosovExists}
Let $\mathcal{T}_1$ and $\mathcal{T}_2$ be distinct 1-vertex triangulations of the torus that do not differ by a 2-2 Pachner move. Then there is a linear Anosov homeomorphism $\phi$ such that $\phi(\mathcal{T}_1) = \mathcal{T}_2$. Moreover, the axis of $\phi$ in the Farey tree runs through the vertices dual to $\mathcal{T}_1$ and $\mathcal{T}_2$.
\end{proposition}

\begin{proof}
Let $a_1$, $b_1$, $c_1$ and $a_2, b_2, c_2$ in $\QQ\cup \{1/0\}$ correspond to the slopes of $\mathcal{T}_1$ and $\mathcal{T}_2$ in the Farey tesselation, where the edge $b_1c_1$ and $b_2 c_2$ are closest to each other. Choose the labelling so that both $a_1$, $b_1$, $c_1$ and $a_2, b_2, c_2$ appear in a clockwise fashion around the circle at infinity. Since $\mathcal{T}_1$ and $\mathcal{T}_2$ do not differ by a 2-2 Pachner move, $b_2 \not= c_1$ or $c_2 \not= b_1$, say $b_2 \not= c_1$. By \reflem{ConstructTorusHomeo}, there is a homeomorphism $\phi$ such that $\phi(a_1) = c_2$, $\phi(b_1) = a_2$ and $\phi(c_1) = b_2$. It is orientation-preserving, since it preserves the cyclic ordering of the vertices. 

Let $\alpha$ be the arc in the Farey tree joining the centre of $a_1b_1c_1$ to the centre of $a_2b_2c_2$. The first edge $e$ of the Farey graph that it crosses is $b_1 c_1$. Then $\phi(e)$ is $a_2 b_2$. This is different from the edge $b_2 c_2$ that $\alpha$ crosses. Note also that $e$ and $\phi(e)$ do not share a vertex. So by \reflem{Anosov}, it is linear Anosov, with axis as claimed.
\end{proof}

\subsection{Splitting sequences between two ideal triangulations}

\begin{theorem} \label{Thm:DualTrainTracks}
Let $\mathcal{T}$ and $\mathcal{T}'$ be distinct ideal triangulations of the once-punctured torus.
Then there are (filling) train tracks $\tau$ and $\tau'$ dual to $\mathcal{T}$ and $\mathcal{T}'$ and a sequence of splits taking $\tau$ to $\tau'$ of length $d_{\Tr(T^2)}(\calT, \calT')$. 
These train tracks are birecurrent, and the set of curves carried by $\tau$ fill the once-punctured torus, as do the set of curves carried by $\tau'$.
\end{theorem}

\begin{proof}
First observe that if $\calT$ and $\calT'$ differ by a 2-2 Pachner move, then it is straightforward to realise their dual trees as train tracks $\tau$ and $\tau'$ in the once-punctured torus that differ by a single split.

So we assume that $\mathcal{T}$ and $\mathcal{T}'$ do not differ by a 2-2 Pachner move.
The triangulations $\mathcal{T}$ and $\mathcal{T}'$ correspond to vertices $v$ and $v'$ of the Farey tree. By Proposition \ref{Prop:PseudoAnosovExists},
there is a linear Anosov homeomorphism $\phi$ of the torus taking $v$ to $v'$. Moreover, the axis of $\phi$ goes through $v$ and $v'$. This has a stable lamination
$\mathcal{L}$ with a transverse measure. We may isotope $\mathcal{L}$ so that it intersects each triangle of $\mathcal{T}$ in normal arcs. The edges of $\mathcal{T}$ then inherit a transverse measure. There are three possible normal arc types in each triangle, and some triangle must be missing an arc type, as otherwise $\mathcal{L}$ would contain a simple closed curve encircling the puncture. Thus, in that triangle, the three edges have measures $a$, $b$ and $a+b$ for some non-negative real numbers $a$ and $b$. As this is the torus, these are the three edges of the other triangle of $\mathcal{T}$, and hence this triangle is also missing an arc type. Now in fact, $a$ and $b$ must both be positive, as otherwise $\mathcal{L}$ would be a thickened simple closed curve. So, the weights on the edges determine a train track $\tau = \tau_0$ that is dual to $\mathcal{T}$ and that carries $\mathcal{L}$. See Figure~\ref{Fig:TrainTrackTorus}. Let $\mu_0$ be the transverse measure on $\tau_0$.

\begin{figure}
  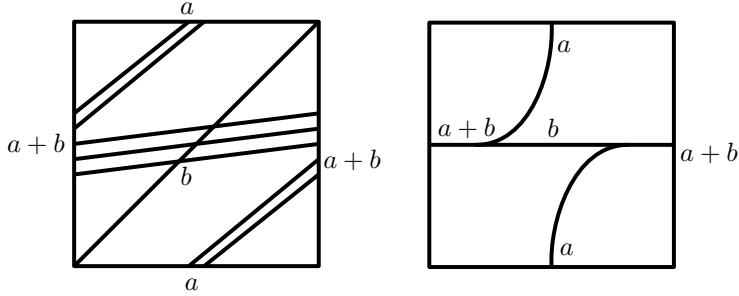
  \caption{Left: A simple closed curve on a triangulated torus. Right: The associated measured train track}
  \label{Fig:TrainTrackTorus}
\end{figure}

We can view $\phi$ as specifying a pseudo-Anosov homeomorphism of the once punctured torus, where its stable lamination is again $\mathcal{L}$.
We now apply Agol's result, \refthm{Agol},
which provides a splitting sequence, giving a sequence of transversely measured train tracks $(\tau_i, \mu_i)$ starting at $(\tau_0, \mu_0)$. A split does not increase the number of complementary regions of a train track. Hence, each train track $\tau_i$ has a single complementary region that is an annular neighbourhood of the puncture. It is therefore dual to an ideal triangulation $\mathcal{T}_i$. Thus, this sequence of train tracks produces an injective path in the Farey tree starting at $v$. Since it is eventually periodic, at some point, this path must land on the axis of $\phi$ and follow this axis from then onwards. However, $v$ is already on the axis of $\phi$. Thus, this path just follows the axis. The axis goes through $v'$, and so when the path reaches $v'$, the result is a train track dual $\tau'$ to $\mathcal{T}'$. Thus, the required splitting sequence has been produced.

We now show that $\tau$ and $\tau'$ are birecurrent. Let $\tau''$ be the train track that is dual to the ideal triangulation with edges having slopes $1/0$, $0/1$ and $1/1$, and where the latter is dual to the large branch. There is a homeomorphism of the once-punctured torus taking the ideal triangulation dual to $\tau$ to the one with edges $1/0$, $0/1$ and $1/1$, and taking the edge dual to the large branch of $\tau$ to $1/1$. Thus, there is a homeomorphism taking $\tau$ to $\tau''$. Similarly, there is a homeomorphism taking $\tau'$ to $\tau''$. So it suffices to show that $\tau''$ is birecurrent. But the simple closed curves with slopes $1/0$, $0/1$ and $1/1$ can be arranged to intersect the branches of $\tau''$ in the required way, thereby establishing that $\tau''$ is transversely recurrent. Also, $\tau''$ is recurrent, since for each of its three branches, there is an obvious simple closed curve carried by $\tau''$ that runs over this branch. Hence, $\tau''$ is birecurrent, as required.

Finally, the curves carried by $\tau''$ fill the once-punctured torus, since they include $1/0$ and $0/1$. Hence, the curves carried by $\tau$ also fill the once-punctured torus, as do the curves carried by $\tau'$.
\end{proof}

%%%%%%%%%%%%%%%%%%%%%%%%%%%%%%%%%%%%%%%%%%%%%%%%%%%%%%%%%%%%%%%%%
\section{Branched covers of the torus}
\label{Sec:BrancherCover}

We will consider branched covering maps $p \colon S \rightarrow T$, where $T$ is the torus and $S$ is a closed orientable surface. We will require that there is a single branch point $b$ in $T$. Our first result says that any 1-vertex triangulation $\mathcal{T}$ of $T$ has a well-defined lift to $S$.

\begin{lemma}
Let $p \colon S \rightarrow T$ be a branched cover of the torus $T$, branched over a single point $b$ in $T$. Let $\mathcal{T}$ and $\mathcal{T}'$ be isotopic 1-vertex triangulations of the torus $T$, with their vertices both equal to $b$. Then their inverse images $\widetilde{\mathcal{T}}$ and $\widetilde{\mathcal{T}}'$ in $S$ are isotopic.
\end{lemma}

\begin{proof} The triangulations $\mathcal{T}$ and $\mathcal{T}'$ are isotopic, but the isotopy is not assumed to preserve basepoints. The isotopy is a 1-parameter family of homeomorphisms, the final one being a homeomorphism $h \colon (T,b) \rightarrow (T,b)$. The Birman exact sequence \cite[Section 4.2.3]{FarbMargalit} for the torus gives that the natural map $\mathrm{MCG}(T, \hbox{1 point}) \rightarrow \mathrm{MCG}(T)$ is an isomorphism. Here, $\mathrm{MCG}(T, \hbox{1 point})$ denotes the group of orientation-preserving homeomorphisms of $T$ that fix a specific point, up to isotopies that fix this point throughout. Hence, the homeomorphism $h$ is isotopic to the identity, via an isotopy that keeps $b$ fixed throughout. This isotopy lifts to an isotopy of $S$ that keeps $p^{-1}(b)$ fixed throughout. This isotopy takes $\widetilde{\mathcal{T}}$ to $\widetilde{ \mathcal{T}}'$.
\end{proof}

As a consequence of the above lemma, it makes sense to compare distances in $\mathrm{Tr}(T)$ with distances in suitable triangulation graphs for $S$.

\begin{theorem} 
\label{Thm:BranchedCoverQI}
Let $p \colon S \rightarrow T$ be a branched cover of the torus, branched over a single point $b$ in $T$, with finite degree $\mathrm{deg}(p)$.
Suppose that the branching index around each point in $p^{-1}(b)$ is greater than $1$. Let $\mathcal{T}_1$ and $\mathcal{T}_2$ be 1-vertex triangulations of the torus $T$ with vertex at $b$, and let $\widetilde{\mathcal{T}}_1$ and $\widetilde{\mathcal{T}}_2$ be their inverse images in $S$. Then there are constants $k_1, k_2 > 0$, depending only on $S$, such that
\[ k_1 \, d_{\mathrm{Tr}(T)} (\mathcal{T}_1, \mathcal{T}_2) - k_2
\leq d_{\mathrm{Tr}(S; -\chi(S))} (\widetilde{\mathcal{T}}_1, \widetilde{\mathcal{T}}_2)
\leq \mathrm{deg}(p) d_{\mathrm{Tr}(T)} (\mathcal{T}_1, \mathcal{T}_2). \]
\end{theorem}

\begin{proof}
The upper bound follows immediately from the fact that a 2-2 Pachner move on a triangulation of $T$ with vertex at $b$ induces $\mathrm{deg}(p)$ 2-2 Pachner moves on the corresponding triangulation of $S$.

So, we focus on the other inequality. We may assume that $\mathcal{T}_1$ and $\mathcal{T}_2$ do not differ by a 2-2 Pachner move, for otherwise
$ d_{\mathrm{Tr}(T)} (\mathcal{T}_1, \mathcal{T}_2) = 1$, and the inequality is trivial.

We view $\mathcal{T}_1$ and $\mathcal{T}_2$ as ideal triangulations of the once-punctured torus. By \refthm{DualTrainTracks}, they are dual to filling birecurrent train tracks $\tau_1$ and $\tau_2$, and there is a splitting sequence taking $\tau_1$ to $\tau_2$ such that the length of this sequence is $d_{\mathrm{Tr}(T)} (\mathcal{T}_1, \mathcal{T}_2)$.
View these as pre-tracks in the torus disjoint from the branch point $b$. Their inverse images $\tilde{\tau}_1$ and $\tilde{\tau}_2$ in $S$ are pre-tracks. In fact, they are train tracks, because each complementary region is a regular neighbourhood of a point of $p^{-1}(b)$. Since the branching index around this point is greater than $1$, the inverse image of the two cusps around $b$ is at least four cusps.
Then the index of the complementary region is at most $\chi(\mbox{disc})-\half(4) \leq -1$. 
Hence the number of compementary regions of $\tilde{\tau}_1$ and $\tilde{\tau}_2$ is at most $-\chi(S)$. So their duals
$\widetilde\calT_1$ and $\widetilde\calT_2$ each have at most $-\chi(S)$ vertices.

The splitting sequence from $\tau_1$ to $\tau_2$ lifts to a splitting sequence from $\tilde{\tau}_1$ to $\tilde{\tau}_2$, of length $\deg(p) d_{\Tr(T^2)}(\calT_1, \calT_2)$. 
We wish to use \refthm{QuasiGeodesicSplittingSequence} to show that the number of splits gives a lower bound on $d_{\Tr(S;-\chi(S))}(\widetilde{\calT}_1,\widetilde{\calT}_2)$, up to multiplicative error depending only on the Euler characteristic of $S$. 
We need to check the hypotheses of \refthm{QuasiGeodesicSplittingSequence}.

We first show that $\tilde{\tau}_1$ is transversely recurrent. Consider any branch $\tilde{e}$ of $\tilde{\tau}_1$. It projects to a branch $e$ of $\tau_1$. Since $\tau_1$ is transversely recurrent, there is a simple closed curve $C$ through $e$ that intersects $\tau_1$ transversely and where $\tau \cup C$ has no bigon complementary region. Let $\tilde C$ be the inverse image of $C$ in $S$. The component of $\tilde C$ going through $\tilde{e}$ establishes the transverse recurrence of $\tilde{\tau}_1$. The same argument establishes that $\tilde{\tau}_2$ is transversely recurrent.

We now show that $\tilde{\tau}_1$ is recurrent. For each branch $\tilde{e}$ of $\tilde{\tau}_1$, let $e$ be its image branch in $\tau_1$. There is a curve $C$ carried by $\tau_1$ running over $e$. Its inverse image in $S$ is a collection $\tilde C$ of curves carried by $\tilde{\tau}_1$. One of these runs over $\tilde e$ as required.

We now show that the curves carried by $\tilde{\tau}_1$ fill $S$.
Let $C$ be a finite collection of curves carried by $\tau_1$ that fill the punctured torus. Place the curves of $C$ in minimal position with respect to each other, in the sense that no two of them have a bigon complementary region. Then the complement of these curves in $T$ is a union of discs. Let $\tilde C$ be the inverse image of these curves in $S$. These are carried by $\tilde{\tau}_1$. They are in minimal position. Their complement in $S$ is a union of discs. Hence, $\tilde C$ fills $S$.

Thus, the hypotheses of \refthm{QuasiGeodesicSplittingSequence} hold and so $d_{\mathrm{Tr}(S; -\chi(S))} (\widetilde{\mathcal{T}}_1, \widetilde{\mathcal{T}}_2)$ is at least the number of splits in a sequence taking $\tilde{\tau}_1$ to $\tilde{\tau}_2$, up to bounded multiplicative error with bound depending only on $S$. We know from above that this number of splits is $\deg(p) d_{\mathrm{Tr}(T)} (\mathcal{T}_1, \mathcal{T}_2)$. It follows that $d_{\mathrm{Tr}(S; -\chi(S))} (\widetilde{\mathcal{T}}_1, \widetilde{\mathcal{T}}_2)$ is at least
$k_1\, d_{\mathrm{Tr}(T)} (\mathcal{T}_1, \mathcal{T}_2) - k_2$, for some constants $k_1, k_2 > 0$ depending only on $S$.
\end{proof}

Suppose $p\from S\to T$ is a branched cover of the torus, branched over a single point $b$ with degree $\deg(p)$. Suppose $\Gamma$ is a spine for $T$ disjoint from $b$. Then observe that the inverse image of $\Gamma$ in $S$ might not be a spine, as its complement consists of at most $\deg(p)$ discs. However, a spine can be formed by removing at most $\deg(p)-1$ edges.

\begin{corollary}\label{Cor:SpineDistanceBranchedCover}
Let $p \colon S \rightarrow T$ be a branched cover of the torus, branched over a single point $b$, with finite degree $\mathrm{deg}(p)$. Suppose that the branching index around each point in $p^{-1}(b)$ is greater than $1$. Let $\Gamma_1$ and $\Gamma_2$ be spines for $T$ that are disjoint from $b$, and let $\widetilde \Gamma_1$ and $\widetilde \Gamma_2$ be their inverse images in $S$. Remove at most $\mathrm{deg}(p) - 1$ edges from each of $\widetilde \Gamma_1$ and $\widetilde{\Gamma}_2$ to form spines $\widetilde{\Gamma}_1'$ and $\widetilde{\Gamma}_2'$ for $S$. Then there are constants $c_1,C_1,c_2, C_2 > 0$, depending only on $p$, such that
\[ c_1 \, d_{\mathrm{Sp}(T)} (\Gamma_1, \Gamma_2) - c_2
\leq d_{\mathrm{Sp}(S)} (\widetilde{\Gamma}_1', \widetilde{\Gamma}_2')
\leq C_1 \, d_{\mathrm{Sp}(T)} (\Gamma_1, \Gamma_2) + C_2. \]
\end{corollary}

\begin{proof}
Note first that it does not matter which edges of $\widetilde \Gamma_1$ and $\widetilde \Gamma_2$ that we remove. For suppose that $\widetilde \Gamma_1''$ and $\widetilde \Gamma_2''$ are other spines also obtained from $\widetilde \Gamma_1$ and $\widetilde{\Gamma}_2$ by removing at most $\mathrm{deg}(p) - 1$ edges from each. Then, $\widetilde \Gamma_1'$ and $\widetilde \Gamma_1''$ differ by at most $24(\mathrm{deg}(p) - 1)$ edge contractions and expansions by \reflem{EdgeSwapBound}, and similarly so do  $\widetilde \Gamma_2'$ and $\widetilde \Gamma_2''$. Thus the difference can be picked up by the constants. 

Define a map $\mathrm{Tr}(T) \rightarrow \mathrm{Tr}(S; -\chi(S))$ first on the vertices: This sends a vertex in $\mathrm{Tr}(T)$, corresponding to a 1-vertex triangulation of $T$ with vertex at $b$, to the vertex in $\mathrm{Tr}(S; -\chi(S))$ corresponding to the triangulation that is the inverse image of $S$. Each edge in $T$ corresponds to a 2-2 Pachner move and this lifts to $\mathrm{deg}(p)$ 2-2 Pachner moves in $S$. Hence, the map $\mathrm{Tr}(T) \rightarrow \mathrm{Tr}(S; -\chi(S))$ can also be defined on edges and is continuous. \refthm{BranchedCoverQI} implies that this map is a quasi-isometry. By \reflem{QIMaps}, the maps $\mathrm{Tr}(T) \rightarrow \mathrm{Sp}(T)$ and $\mathrm{Tr}(S) \rightarrow \mathrm{Sp}(S)$ are quasi-isometries, and the inclusion $\mathrm{Tr}(S) \rightarrow \mathrm{Tr}(S;n)$ is a quasi-isometry for any positive integer $n$.
A quasi-inverse is given by taking any triangulation with at most $n$ vertices, dualising it to form a trivalent graph, then removing edges to form a spine, and then dualising to form a 1-vertex triangulation. This construction can be chosen to be invariant under the mapping class group, and hence forms a quasi-isometry $\mathrm{Tr}(S;n) \rightarrow \mathrm{Tr}(S)$.
The composition
\[ \mathrm{Sp}(T) \rightarrow \mathrm{Tr}(T) \rightarrow  \mathrm{Tr}(S; -\chi(S)) \rightarrow \mathrm{Tr}(S) \rightarrow \mathrm{Sp}(S) \]
is a quasi-isometry. Thus, we obtain the required inequalities.
\end{proof}

\begin{remark}\label{Rmk:BranchedCover}
Theorem~\ref{Thm:BranchedCoverQI} and Corollary~\ref{Cor:SpineDistanceBranchedCover} require a branched cover of the torus with a single branch point $b$ and where each point of $p^{-1}(b)$ has branching index at least two. One such branched cover is obtained by the following process. First, take the $\mathbb{Z}/2 \times \mathbb{Z}/2$ cover of $T^2$ arising from the natural homomorphism $\pi_1(T^2) \rightarrow H_1(T^2; \mathbb{Z}/2)$. Then restrict this to a cover $F \rightarrow T^2 - \{ b \}$, where $F$ is a four-times punctured torus. Next form the cover of $F$ arising from $\pi_1(F) \rightarrow H_1(F; \mathbb{Z}/2)$. Finally complete this to form the required branched cover $S$ of $T^2$.
\end{remark}

%%%%%%%%%%%%%%%%%%%%%%%%%%%%%%%%%%%%%%%%%%%%%%%%%%%%%%%%%%%%%%%%%
\section{Triangulations and handle structures of products} 
\label{Sec:Products}

\begin{named}{\refthm{ProductComplexity}}
Let $\mathcal{T}_0$ and $\mathcal{T}_1$ be 1-vertex triangulations of the torus $T^2$. Let $\Delta(\mathcal{T}_0, \mathcal{T}_1)$ denote the minimal number of tetrahedra in any triangulation of $T^2 \times [0,1]$ that equals $\mathcal{T}_0$ on $T^2 \times \{ 0 \}$ and equals $\mathcal{T}_1$ on $T^2 \times \{ 1 \}$. Then there is a universal constant $k_{\mathrm{prod}} >0$ such that
\[ k_{\mathrm{prod}} \ d_{\mathrm{Tr}(T^2)}(\mathcal{T}_0, \mathcal{T}_1)  \leq \Delta(\mathcal{T}_0, \mathcal{T}_1) \leq d_{\mathrm{Tr}(T^2)}(\mathcal{T}_0, \mathcal{T}_1) + 6. \]
\end{named}

\begin{proof}
The upper bound is straightforward: Let $\alpha$ be a shortest path in $\mathrm{Tr}(T^2)$ from $\mathcal{T}_0$ to $\mathcal{T}_1$. This determines a sequence of 1-vertex triangulations, starting at $\mathcal{T}_0$ and ending at $\mathcal{T}_1$. We use this to build a triangulation of $T^2 \times [0,1]$, as follows. Start with $T^2 \times \{ 0 \}$ triangulated using $\mathcal{T}_0$. Then take the product of this with $[0,1]$ and triangulate each of the resulting prisms using 3 tetrahedra. If chosen correctly, these patch together to form a triangulation of $T^2 \times [0,1]$ with $6$ tetrahedra, where $T^2 \times \{ 0 \}$ and $T^2 \times \{ 1 \}$ are both triangulated using $\calT_0$. Then layer onto $T^2 \times \{ 1\}$ a sequence of tetrahedra, specified by the sequence of Pachner moves, until we reach $\calT_1$. This gives the upper bound.

As for the lower bound, let $p \colon S \rightarrow T^2$ be a branched cover, with single branch point $b$ and where each point of $p^{-1}(b)$ has branching index at least two.
For example, take the branched cover of Remark~\ref{Rmk:BranchedCover}.

Now let $\mathcal{T}$ be a triangulation of $T^2 \times [0,1]$ that equals $\mathcal{T}_0$ on $T^2 \times \{ 0 \}$ and equals $\mathcal{T}_1$ on $T^2 \times \{ 1 \}$ and that realises $\Delta(\mathcal{T}_0, \mathcal{T}_1)$. By \refthm{VerticalArc}, the 23rd iterated barycentric subdivision $\mathcal{T}^{(23)}$ contains an arc in its 1-skeleton that is vertical.
Note $\calT^{(23)}$ consists of $(24)^{23}\Delta(\calT_0,\calT_1)$ tetrahedra.
By Lemma~\ref{Lem:BarycentricPachner}, there is a sequence of at most $4 (1 + 6 + \dots + 6^{22}) < 6^{23}$ Pachner moves taking $\calT_0$ to $\calT_0^{(23)}$. Perform the reverse of this sequence, and realise each Pachner move on the boundary of $S \times [0,1]$ by attaching a tetrahedron to $S \times \{ 0 \}$. Do the same for $S \times \{ 1 \}$. Let $\calT_+$ be the resulting triangulation of $S \times [0,1]$. It has at most $(24)^{23} \Delta(\calT_0, \calT_1) + 2 \cdot 6^{23}$
tetrahedra and it equals $\calT_0$ and $\calT_1$ on its boundary. By Lemma \ref{Lem:MaintainVerticalAfterPachner}, it also contains an arc in its 1-skeleton that is vertical.
Hence, the inverse image of $\mathcal{T}_+$ under the branched covering map is a triangulation of $S \times [0,1]$. It has at most $(24)^{23}\,\deg(p)\Delta(\mathcal{T}_0, \mathcal{T}_1) + 2\cdot 6^{23} \, \deg(p)$ tetrahedra.

Let $\widetilde{\mathcal{T}}_0$ denote the restriction of the triangulation $p^{-1}(\calT_+)$ to $S \times \{ 0 \}$, and let $\widetilde{\mathcal{T}}_1$ denote the triangulation on $S \times \{ 1 \}$.
By \refthm{BranchedCoverQI}, there are constants $k_1,k_2 > 0$, depending only on $p$, such that
\begin{equation}\label{Eqn:TrBound}
  k_1 \, d_{\mathrm{Tr}(T)} (\mathcal{T}_0, \mathcal{T}_1) - k_2
  \leq d_{\mathrm{Tr}(S; -\chi(S))} (\widetilde{\mathcal{T}}_0, \widetilde{\mathcal{T}}_1).
\end{equation}

By \reflem{ToOneVertexTriangulation}, there is a sequence of at most $4 \, \mathrm{deg}(p)$ Pachner moves taking $\widetilde{\mathcal{T}}_0$ to a 1-vertex triangulation $\widetilde{\mathcal{T}}_0'$, and a sequence of at most $4 \, \mathrm{deg}(p)$ Pachner moves taking $\widetilde{\mathcal{T}}_1$ to a 1-vertex triangulation $\widetilde{\mathcal{T}}_1'$. Considering the reverse moves, we obtain
\begin{equation}\label{Eqn:TrSTildeTBound}
d_{\Tr(S;-\chi(S))}(\widetilde{\mathcal{T}}_0, \widetilde{\mathcal{T}}_1) \leq
d_{\Tr(S;-\chi(S))}(\widetilde{\mathcal{T}}_0', \widetilde{\mathcal{T}}_1')+ 8 \,\deg(p).
\end{equation}

Each Pachner move corresponds to the addition of a 3-simplex to the triangulation of $S \times [0,1]$. So, we obtain a triangulation $\mathcal{T}'$ of $S \times [0,1]$ that equals $\widetilde{\mathcal{T}}_0'$ on $S \times \{ 0 \}$, equals $\widetilde{\mathcal{T}}_1'$ on $S \times \{ 1 \}$ and has at most $(24)^{23} \mathrm{deg}(p) \Delta(\calT_0, \calT_1) + (2 \cdot 6^{23} + 8) \, \mathrm{deg}(p)$ tetrahedra.

Let $\Delta(\calT')$ denote the minimal number of tetrahedra in any triangulation of $S\times[0,1]$ that equals $\calT_0'$ on $S\times \{0\}$ and $\calT_1'$ on $S\times\{1\}$. By the above observation,
\begin{equation}\label{Eqn:DeltaBound}
  \Delta(\calT') \leq (24)^{23}\deg(p)\Delta(\calT_0,\calT_1) + (2 \cdot 6^{23} + 8) \,\deg(p).
\end{equation}
By \refthm{TriangulationProductOneVertex}, there is a constant $k_3 > 0$ depending only on $S$ such that 
\begin{equation}\label{Eqn:DeltaLowerBound}
  \Delta(\calT') \geq k_3 \ d_{\mathrm{Tr}(S)} (\widetilde{\mathcal{T}}'_0, \widetilde{\mathcal{T}}'_1).
\end{equation}

By \reflem{QIMaps}, there are constants $k_4, k_5 > 0$ depending only on $S$ and $\mathrm{deg}(p)$ such that
\begin{equation}\label{Eqn:QIBound}
  d_{\mathrm{Tr}(S)} (\widetilde{\mathcal{T}}'_0, \widetilde{\mathcal{T}}'_1) \geq k_4 \, d_{\mathrm{Tr}(S;-\chi(S))} (\widetilde{\mathcal{T}}'_0, \widetilde{\mathcal{T}}'_1) - k_5.
\end{equation}
Putting this all together, we obtain
\begin{align*}
(24)^{23} \mathrm{deg}(p) & \Delta(\calT_0, \calT_1) + (2\cdot 6^{23} + 8) \, \mathrm{deg}(p) \geq \Delta(\calT') & \mbox{by \eqref{Eqn:DeltaBound}}\\
& \geq k_3 \ d_{\mathrm{Tr}(S)} (\widetilde{\mathcal{T}}'_0, \widetilde{\mathcal{T}}'_1) & \mbox{by \eqref{Eqn:DeltaLowerBound}}\\
& \geq k_3 (k_4 \, d_{\mathrm{Tr}(S; -\chi(S))} (\widetilde{\mathcal{T}}'_0, \widetilde{\mathcal{T}}'_1) - k_5) & \mbox{by \eqref{Eqn:QIBound}} \\
& \geq k_3 k_4 (d_{\mathrm{Tr}(S; -\chi(S))} (\widetilde{\mathcal{T}}_0, \widetilde{\mathcal{T}}_1) - 8 \mathrm{deg}(p)) - k_3 k_5 & \mbox{by \eqref{Eqn:TrSTildeTBound}}\\
& \geq k_3 k_4 k_1 d_{\mathrm{Tr}(T)} (\mathcal{T}_0, \mathcal{T}_1) - k_3 k_4 k_2 - k_3 k_5 - 8 k_3 k_4 \mathrm{deg}(p) & \mbox{by \eqref{Eqn:TrBound}}
\end{align*}
This gives a linear lower bound on $\Delta(\calT_0, \calT_1)$ in terms of $d_{\mathrm{Tr}(T)} (\mathcal{T}_0, \mathcal{T}_1)$.
For all but at most finitely many positive values of $d_{\Tr(T)}(\calT_0, \calT_1)$, this lower bound will be positive and implies that there exists $k_{\mathrm{prod}} >0$ such that
\[ \Delta(\calT_0, \calT_1) \geq k_{\mathrm{prod}}  \, d_{\mathrm{Tr}(T)} (\mathcal{T}_0, \mathcal{T}_1).\]
For the remaining values, $\Delta(\calT_0, \calT_1) $ is positive, and so a universal $k_{\mathrm{prod}}  > 0$ can be chosen appropriately.
\end{proof}

In \cite{LackenbyPurcell:Fibred}, given a pre-tetrahedral handle structure $\calH$ on $S\times[0,1]$, we considered the number of edge swaps required to transfer a cellular spine on $S\times\{0\}$ to a cellular spine on $S\times\{1\}$. The main technical theorem of that paper, \cite[Theorem~9.11]{LackenbyPurcell:Fibred}, gives a linear lower bound on the number of edge swaps in terms of $\Delta(\calH)$ under appropriate hypotheses.
We will prove an analogous theorem for tori.

\begin{theorem}
\label{Thm:ProductToriHandles}
Let $\calH$ be a pre-tetrahedral handle structure for $T^2 \times [0,1]$ that admits no annular simplification. Let $\Gamma_0$ be a cellular spine in $T^2 \times \{ 0 \}$. Then there is a sequence of at most $k_{\mathrm{hand}} \, \Delta(\mathcal{H})$ edge swaps taking $\Gamma_0$ to a spine $\Gamma_1$ that is cellular with respect to $T^2 \times \{ 1 \}$. Here, $k_{\mathrm{hand}}$ is a universal constant.
\end{theorem}

\begin{proof}
For ease of notation, set $\Gamma_0 = \Gamma^{(0)}$.
In the course of the proof, we will obtain spines $\Gamma^{(0)}, \Gamma^{(1)} ,\dots, \Gamma^{(n)}$, with each $\Gamma^{(j)}$ obtained from $\Gamma^{(j-1)}$ by at most $k^{(j)}\Delta(\calH)$ edge swaps, where $k^{(j)}$ is a universal constant. The number of spines in the sequence will be universally bounded (by $n=6$), so the result follows.
  
First, let $\calB$ be a maximal generalised parallelity bundle for $\calH$. By \reflem{GenParIncompr}, each component of $\calB$ has incompressible horizontal boundary, and is either an $I$-bundle over a disc or has incompressible vertical boundary. Since $T^2 \times I$ does not contain a properly embedded M\"obius band or Klein bottle,
the only possibilities are that $\calB$ consists of a union of $I$-bundles over discs and annuli, or $\calB$ is all of $T^2\times I$.
We also claim that each component of $\calB$ that is an $I$-bundle over an annulus has one horizontal boundary component in $T^2 \times \{ 0 \}$ and one horizontal boundary component in $T^2 \times \{ 1 \}$. Suppose that on the contrary, there is an $I$-bundle over an annulus with both horizontal boundary components in the same component of $T^2 \times \partial I$. Then its two vertical boundary components are boundary-parallel. Pick such a component of $\calB$ that is outermost in $T^2 \times I$. Let $A'$ be its vertical boundary component that is not outermost. Then $\partial A'$ bounds an annulus $A$ in $T^2 \times \partial I$, and $A \cup A'$ bounds a product region $P$. Hence, $\calH$ admits an
annular simplification, contrary to hypothesis.

Let $\calB'\subset \calB$ denote the parallelity bundle for $\calH$. 
Recall from \refdef{GeneralisedParallelityBundle}~(6), the definition of a generalised parallelity bundle, that the intersection between $\partial_h \calB'$ and each essential component $A''$ of $\partial_h \calB$ contains a component that is equal to $A''$ with some discs $D$ removed from its interior. As a first step, we will adjust $\Gamma_0= \Gamma^{(0)}$ by sliding it off all such discs, using \reflem{SlidingOffDiscs}. 

By \reflem{LengthVerticalBoundary}, the length of $\bdy_v\calB'$ is at most $56\Delta(\calH)$.
Hence, the length of $\partial D$ is at most $112 \Delta(\calH)$. In particular, the number of components of $D$ is at most $112\Delta(\calH)$. Then \reflem{SlidingOffDiscs} implies that after at most $6\cdot112\Delta(\calH) + 2\cdot 112\Delta(\calH) = 896\Delta(\calH)$ edge swaps, we obtain a spine $\Gamma^{(1)}$ whose intersection with $\calB$ lies entirely within the parallelity bundle $\calB'$: $\Gamma^{(1)}\cap \calB \subset \calB'$. 

\medskip

\emph{Case 1.} $\calB$ is all of $T^2 \times [0,1]$.

Then some component of the parallelity bundle $\calB'$ is of the form $(T^2\setminus D)\times [0,1]$, where $D$ is a union of disjoint discs in $T^2$. In particular, $\Gamma^{(1)}$ lies in parallelity handles that run from $T^2\times\{0\}$ to $T^2\times\{1\}$ and that respect the product structure of $T^2 \times [0,1]$. Transfer $\Gamma^{(1)}$ to $T^2\times \{1\}$ using the product structure on $\calB'$, obtaining a spine $\Gamma^{(2)}$ in $T^2\times\{1\}$ without any additional edge swaps. Observe this is cellular in the cell structure associated with $\calH$, so set $\Gamma_1=\Gamma^{(2)}$. The proof is complete in this case. 

\medskip

\emph{Case 2.} $\calB$ consists of $I$-bundles over discs and annuli.

In this case, we cannot ensure that all of $\Gamma^{(1)}$ lies only in the parallelity bundle, and so we cannot transfer as simply as in the previous case. Instead, we will obtain a triangulation from a simplified handle structure and apply \refthm{ProductComplexity}.

First we adjust $\Gamma^{(1)}$ further. Again \reflem{LengthVerticalBoundary} implies that the length of $\bdy_v\calB$ is at most $56\Delta(\calH)$, and there are at most $112\Delta(\calH)$ components of $\bdy_h\calB$. 
We will now apply \reflem{SlidingOffAnnuli}. Adjust $\Gamma^{(1)}$ to a new spine $\Gamma^{(2)}$ that is disjoint from the interior of the disc components of $\bdy_h\calB$, and intersects the interior of each annular component of $\bdy_h\calB$ in at most one arc. Moreover, this arc is a subset of $\Gamma^{(1)}$ and so the arc lies in the parallelity bundle $\calB'$. By \reflem{SlidingOffAnnuli}, $\Gamma^{(2)}$ is obtained from $\Gamma^{(1)}$ from a number of edge swaps bounded by
\[ 6\cdot 112\Delta(\calH) + 16 + 2\cdot 112\Delta(\calH) + 2\cdot 112\Delta(\calH) \leq 10\cdot 112\Delta(\calH) + 16 \leq 1248\Delta(\calH). \]
Here we used that $\Delta(\calH)\geq1/8$, since otherwise $\Delta(\calH) = 0$ and the parallelity bundle for $\calH$ is then all of $\calH$, which is dealt with in Case 1.

Form a new 
handle structure $\calH'$ as follows. Replace each component of $\calB$ that is an $I$-bundle over a disc by a single 2-handle. Replace each component of $\calB$ that is an $I$-bundle over an annulus by a 1-handle and a 2-handle, arranged such that the intersection of these two handles contains an arc of $\Gamma^{(2)}$ within the annulus, if there is one. This is possible because each component of $\calB$ that is an $I$-bundle over an annulus intersects $T^2 \times \{ 0 \}$ in a single component.
Then $\Gamma^{(2)}$ remains a cellular spine in this new handle structure $\calH'$. Note $\calH'$ has no parallelity 0-handles.

However, note that $\calH'$ may no longer be pre-tetrahedral. For each component of $\calB$ that is not a 2-handle, its vertical boundary consists of alternating components of intersection with 1-handles and 2-handles, and these must be adjacent to semi-tetrahedral 0-handles, attached where a 1-handle is bounded by exactly two 2-handles. The process of replacing a $D^2\times I$ component of $\calB$ replaces one or both instance of such a 1-handle and its adjacent 2-handles with a single 2-handle. This gives a clipped semi-tetrahedral 0-handle as in \reffig{ClippedSemiTet}. Similarly clipped semi-tetrahedral 0-handles may arise when replacing components that are $I$-bundles over an annulus, but these are the only adjustments that need to be made.

Recall that we may still define the complexity of a handle structure that is pre-tetrahedral aside from a finite number of clipped semi-tetrahedral 0-handles. Since clipped semi-tetrahedral handles contribute the same as semi-tetrahedral handles to complexity, we have $\Delta(\calH') \leq \Delta(\calH)$. 

Let $\calT$ be the triangulation obtained from $\calH'$ as in \reflem{DualToPreTetHS-Clipped}. This satisfies $\Delta(\calT) \leq 1152 \Delta(\calH') \leq 1152\Delta(\calH)$. Let $\calT_0$ and $\calT_1$ be the induced triangulations of $T^2\times\{0\}$ and $T^2\times\{1\}$. Observe that our spine $\Gamma^{(2)}$ is a subcomplex of $\calT_0$. 

By \reflem{ToOneVertexTriangulation}, there is a sequence of at most $16 \Delta(\calT)$ Pachner moves taking $\calT_0$ and $\calT_1$ to 1-vertex triangulations $\calT_0'$ and $\calT_1'$. 
By \reflem{PachnerMoveSwap}, there is a sequence of at most $16\Delta(\calT) \leq 16\cdot 1152\Delta(\calH)$ edge swaps taking $\Gamma^{(2)}$ to a subcomplex $\Gamma^{(3)}$ of $\calT'_0$.

Realise the Pachner moves as tetrahedra, to form a triangulation $\calT'$ of $T^2 \times [0,1]$ with triangulations $\calT_0'$ and $\calT_1'$ on $T^2 \times \{ 0 \} $ and $T^2 \times \{ 1 \}$, respectively, and with $\Delta(\calT') \leq 17 \Delta(\calT)$.

By \refthm{ProductComplexity}, $d_{\mathrm{Tr}(T^2)}(\calT'_0, \calT'_1)$ is at most $k_{\mathrm{prod}} \, \Delta(\calT')$ for some universal constant $k_{\mathrm{prod}} >0$.
That is, there is a sequence of at most $k_{\mathrm{prod}} \Delta(\calT')\leq 17k_{\mathrm{prod}} \Delta(\calT)$ Pachner moves taking $\calT_0'$ to $\calT_1'$. Follow this by at most $16\Delta(\calT)$ Pachner moves taking $\calT_1'$ to $\calT_1$, for a total of $(17k_{\mathrm{prod}}+16)\Delta(\calT)$ Pachner moves taking $\calT_0'$ to $\calT_1$. 
Again \reflem{PachnerMoveSwap} implies there is a sequence of at most $(17k_{\mathrm{prod}} + 16) \Delta(\calT') \leq k \Delta(\calH)$ edge swaps taking the spine $\Gamma^{(3)}$ to a subcomplex $\Gamma^{(4)}$ of $\calT_1$, where $k>0$ is a universal constant.

The 1-skeleton of the cell structure associated with $\calH'$ on $T^2\times\{1\}$ is a subcomplex of $\calT_1$, so the next step is to adjust $\Gamma^{(4)}$ to lie only in this subcomplex. The 2-cells of the cell structure are discs. The total length of their boundary is at most the total length of the 1-skeleton of the tetrahedra of $\calT$, which is at most $6\Delta(\calT)$. The number of discs is at most the number of triangles in $\calT$, which is at most $4\Delta(\calT)$. Hence by \reflem{SlidingOffDiscs}, we may modify $\Gamma^{(4)}$ to a spine $\Gamma^{(5)}$ that is cellular with respect to $\calH'$ using at most $36\Delta(\calT)$ edge swaps.

The cell structure that $T^2\times\{1\}$ inherits from $\calH'$ agrees with that of $\calH$ away from $\bdy_h\calB$. Within the interior of each essential annular component of $\partial_h \calB$, the structure of $\calH'$ consists of two 1-cells and two 2-cells. One of these 1-cells is the arc of intersection with $\Gamma^{(2)}$, which we arranged to be part of the parallelity bundle.
Hence, it is also cellular in the cell structure that
inherits from $\calH$. If necessary, a single edge swap takes $\Gamma^{(5)}$ to a spine intersecting the interior of this annulus just in this arc. As there are at most $56 \Delta(\calH)$ annular components of $\partial_h \calB$ in $T^2 \times \{ 1 \}$, this can be done for all annuli with at most $56 \Delta(\calH)$ additional edge swaps, obtaining a spine $\Gamma^{(6)}$. Now $\Gamma^{(6)}$ is cellular with respect to $\calH$, so we set this equal to $\Gamma_1$.

In summary, starting with the spine $\Gamma_0=\Gamma^{(0)}$, we have found a sequence of edge swaps taking the spine $\Gamma^{(i)}$ to a spine $\Gamma^{(i+1)}$, for $i=0, \dots, 5$, where $\Gamma^{(6)}=\Gamma_1$ is the desired cellular spine on $T^2\times\{1\}$, such that the number of edge swaps required to take $\Gamma^{(i)}$ to $\Gamma^{(i+1)}$ is bounded by a uniform constant times $\Delta(\calH)$, for each $i=1, \dots, 5$. 
\end{proof}

%%%%%%%%%%%%%%%%%%%%%%%%%%%%%%%%%%%%%%%%%%%%%%%%%%%%%%%%%%%%%%%%%
\section{Triangulations of sol manifolds}
\label{Sec:Sol}

In this section, we prove Theorems~\ref{Thm:SolAlternative},~\ref{Thm:Sol} and~\ref{Thm:FibredTori}.

\begin{lemma}\label{Lem:UpperSol}
  Let $\phi\from T^2\to T^2$ be a linear Anosov homeomorphism. Then the triangulation complexity satisfies
  \[ \Delta( (T^2\times I)/\phi ) \leq 6 + \ell_{\Tr(T^2)}(\phi). \]
\end{lemma}

\begin{proof}
There is a 1-vertex triangulation $t$ of the torus such that the distance in the Farey tree between $t$ and $\phi(t)$ realises the translation length $\ell_{\Tr(T^2)}(\phi)$. As in the proof of \refthm{ProductComplexity}, we triangulate $T^2 \times [0,1]$ using $6 + \ell_{\Tr(T^2)}(\phi)$ tetrahedra, with $T^2\times\{0\}$ triangulated using $t$ and $T^2\times\{1\}$ triangulated using $\phi(t)$. Then glue top to bottom using $\phi$. 
The result is a triangulation of  $(T^2\times I)/\phi$ with $6 + \ell_{\Tr(T^2)}(\phi)$ tetrahedra.
\end{proof}

\begin{lemma}\label{Lem:LowerSol}
  Let $\phi\from T^2\to T^2$ be a linear Anosov homeomorphism. Then there exists a universal constant $k'_{\mathrm{sol}}>0$ such that
  \[ k'_{\mathrm{sol}} \, \ell_{\Tr(T^2)}(\phi) \leq \Delta( (T^2\times I)/\phi ).\]
\end{lemma}

\begin{proof}
Let $M=(T^2\times I)/\phi$. We will show that $\Delta(M)$ is at least a constant times the translation distance of $\phi$ in the spine graph $\mathrm{Sp}(T^2)$. As $\mathrm{Sp}(T^2)$ and ${\mathrm{Tr}(T^2)}$ are quasi-isometric, by \reflem{QIMaps}, this will prove the result.

Consider a triangulation $\calT$ for $M$ with $\Delta(\calT) = \Delta(M)$. Let $S$ be a normal fibre in $M$ with least weight. This corresponds to a surface,
also called $S$, that intersects each handle in the dual handle structure in a collection of properly embedded discs. This surface $S$ inherits a cell structure in which each of these discs is a 2-cell.
Pick some spine $\Gamma$ for $S$ that is cellular. Let $\calH$ be the handle structure that results from cutting $M$ along $S$. Then by \reflem{ComplexityUnderCutting}, $\Delta(\calH) = \Delta(\calT)$.

Since $S$ has least weight in its isotopy class, $\calH$ does not admit any annular simplifications, by \refthm{ExtendToGenParBdle}.
So, by \refthm{ProductToriHandles}, there is a sequence of at most $k_{\mathrm{hand}} \Delta(\calH)$ edge contractions and expansions taking $\Gamma$ in $T^2 \times \{ 0 \}$ to a spine $\Gamma_1$ in $T^2 \times \{ 1 \}$ that is cellular. Now apply the gluing map $\phi$ between $T^2 \times \{ 1 \}$ and $T^2 \times \{ 0 \}$ to get the spine $\phi(\Gamma_1)$ in $T^2 \times \{ 0 \}$.
Next apply \refthm{ProductToriHandles} again, to obtain a sequence of at most $k_{\mathrm{hand}} \Delta(\calH)$ edge contractions and expansions taking $\phi(\Gamma_1)$ in $T^2 \times \{ 0 \}$ to a cellular spine $\Gamma_2$ in $T^2 \times \{ 1 \}$. Keep repeating this process, giving a sequence of spines $\Gamma_i$ that are cellular in $T^2 \times \{ 1 \}$. Thus, the distance in $\mathrm{Sp}(T^2)$ between $\phi(\Gamma_i)$ and $\Gamma_{i+1}$ is at most $k_{\mathrm{hand}} \Delta(\calH)$. There are only finitely many cellular spines in  $T^2 \times \{ 1 \}$ and so there are integers $r < s$ such that $\Gamma_r = \Gamma_s$. By relabelling, we may assume that $r=0$ and $s=n$, say. Thus, with respect to the metric $d$ on $\mathrm{Sp}(T^2)$, we have the following inequalities:
\begin{align*}
d(\phi^n\Gamma_0,\Gamma_0) &= d(\phi^n\Gamma_0, \Gamma_n) \\
&\leq d(\phi^n \Gamma_0, \phi^{n-1}\Gamma_1) + d(\phi^{n-1}\Gamma_1,\phi^{n-2}\Gamma_2) + \dots + d(\phi\Gamma_{n-1},\Gamma_n) \\
&= d(\phi \Gamma_0, \Gamma_1) + d(\phi\Gamma_1, \Gamma_2) + \dots + d(\phi\Gamma_{n-1},\Gamma_n) \\
&\leq k_{\mathrm{hand}} n \Delta(\calH).
\end{align*}
So, the translation length $\ell_{\Sp(T^2)}(\phi^n)$ of $\phi^n$ is at most $k_{\mathrm{hand}} n \Delta(\calH)$. But $\ell_{\Sp(T^2)}(\phi^n)$ is $n$ times the translation length $\ell_{\Sp(T^2)}(\phi)$, since $\phi$ acts on the tree $\mathrm{Sp}(T^2)$ by translation along an axis; see Lemma~\ref{Lem:StableTransTorus}.
Therefore, the translation length of $\phi$ is at most $k_{\mathrm{hand}} \Delta(\calH) = k_{\mathrm{hand}} \Delta(M)$.
\end{proof}

Recall from Section \ref{SubSec:PLS2Z} that $\mathrm{PSL}(2, \mathbb{Z})$ is isomorphic to $\mathbb{Z}_2 \ast \mathbb{Z}_3$ where the factors are generated by 
\[
S = 
\left(
\begin{matrix}
0 & -1 \\
1 & 0 \\
\end{matrix}
\right)
\qquad
T = 
\left(
\begin{matrix}
0 & -1 \\
1 & -1 \\
\end{matrix}
\right).
\]

\begin{named}{\refthm{SolAlternative}}
Let $A$ be an element of $\mathrm{SL}(2, \mathbb{Z})$ with $|\mathrm{tr}(A)| > 2$. Let $M$ be the sol 3-manifold $(T^2 \times [0,1]) / (A(x,1) \sim (x,0))$. Let $\overline{A}$ be the image of $A$ in 
$\mathrm{PSL}(2, \mathbb{Z})$ and let $\ell(\overline{A})$ be the length of a cyclically reduced 
word in the generators $S$ and $T^{\pm 1}$ that is conjugate to $\overline{A}$. Then, there is a universal constant $k_{\mathrm{sol}}>0$ such that
\[ k_{\mathrm{sol}} \ell(\overline{A}) \leq \Delta(M) \leq (\ell(\overline{A})/2) + 6.\]
\end{named}

\begin{proof}
As explained in Section \ref{SubSec:PLS2Z}, the Farey tree is closely related to the Cayley graph of $\mathbb{Z}_2 \ast \mathbb{Z}_3$ with respect to the generators $S$ and $T$. Specifically this Cayley graph is obtained from the Farey tree as follows: replace each vertex of the tree by a triangle, with each edge oriented and labelled by $T$; replace each edge of the Farey tree by two edges, both labelled by $S$ and pointing in opposite directions. The element $A$ in $\mathrm{SL}(2, \mathbb{Z})$ acts on this Cayley graph as it does on the Farey tree, and  the translation lengths of these two actions differ by a factor of $2$. Indeed, the invariant axis in the Farey tree can be used to produce an invariant geodesic in the Cayley graph, and the translation length along this geodesic is twice that of  the translation length along the axis in the Farey tree. As one travels along this geodesic, one reads off a word in $S$, $T$ and $T^{-1}$ which is a cyclically reduced representative for a conjugate of 
$\overline{A}$. Thus its length $\ell(\overline{A})$ is twice the translation length of the action of $A$ on the Farey tree. By Lemmas \ref{Lem:LowerSol} and \ref{Lem:UpperSol}, this is, up to a bounded multiplicative factor, the triangulation complexity of $M$.
\end{proof}

\begin{named}{\refthm{Sol}}
Let $A$ be an element of $\mathrm{SL}(2, \mathbb{Z})$ with $|\mathrm{tr}(A)| > 2$. Let $\overline{A}$ be the image of $A$ in $\mathrm{PSL}(2, \mathbb{Z})$. Suppose that $\overline{A}$ is $B^n$ for some positive integer $n$ and some $B \in \mathrm{PSL}(2, \mathbb{Z})$ that cannot be expressed as a proper power. Let $M$ be the sol 3-manifold $(T^2 \times [0,1]) / (A(x,1) \sim (x,0))$. Let $[a_0, a_1, \dots]$ be the continued fraction expansion of $\sqrt{\tr(A)^2 - 4}$ where $a_i$ is positive for each $i >0$ and let $(a_r, \dots, a_s)$ denote its periodic part. Then there is a universal constant $k'_{\mathrm{sol}}>0$ such that
\[ k'_{\mathrm{sol}} n \sum_{i=r}^s a_i \leq \Delta(M) \leq 6 + n \sum_{i=r}^s a_i. \]
\end{named}

\begin{proof}
Let $\phi\from T^2\to T^2$ be the homeomorphism determined by $A$.
By \reflem{AnosovTrace}, $\phi$ is linear Anosov, so by Lemmas~\ref{Lem:UpperSol} and~\ref{Lem:LowerSol}, we have
\[ k'_{\mathrm{sol}} \ell_{\Tr(T^2)}(\phi) \leq \Delta(M) \leq 6 + \ell_{\Tr(T^2)}(\phi).\]
Now by \refthm{TrIsFarey}, $\ell_{\Tr(T^2)}(\phi)$ is the translation length of $\phi$ in the Farey tree. 
By \refprop{AnosovDistance}, this translation length is $n\sum_{i=r}^s a_i$.
\end{proof}

\begin{named}{\refthm{FibredTori}}
Let $\phi\from T^2 \to T^2$ be a linear Anosov homeomorphism. Then the following quantities are within bounded ratios of each other:
\begin{enumerate}
\item the triangulation complexity of $(T^2 \times I)/ \phi$;
\item the translation distance (or stable translation distance) of $\phi$ in the thick part of the Teichm\"uller space of $T^2$;
\item the translation distance (or stable translation distance) of $\phi$ in the mapping class group of $T^2$;
\item the translation distance (or stable translation distance) of $\phi$ in $\mathrm{Tr}(T^2)$.
\end{enumerate}
In (3), we metrise $\mathrm{MCG}(T^2)$ by fixing the finite generating set \[
\left(
\begin{matrix}
1 & 1 \\
0 & 1 \\
\end{matrix}
\right),
\quad
\left(
\begin{matrix}
1 & 0 \\
1 & 1 \\
\end{matrix}
\right).
\]
\end{named}

\begin{proof}
The fact that the quantities (2), (3) and (4) are within a bounded ratio of each other is a rapid consequence of the Milnor-\u{S}varc lemma, as explained in \cite[Section~1.2]{LackenbyPurcell:Fibred}.
The relationship between (1) and (4) now follows immediately from Lemmas~\ref{Lem:UpperSol} and~\ref{Lem:LowerSol}.
\end{proof}

%%%%%%%%%%%%%%%%%%%%%%%%%%%%%%%%%%%%%%%%%%%%%%%%%%%%%%%%%%%%%%%%%
\section{Triangulations of lens spaces}
\label{Sec:Lens}

The following was one of the main theorems of \cite{LackenbySchleimer}.

\begin{theorem}
\label{Thm:LensSpaceCore}
Let $M$ be a lens space other than a prism manifold $L(4p, 2p\pm1)$ or $\mathbb{RP}^3$.
Let $\mathcal{T}$ be any triangulation of $M$. Then the iterated barycentric subdivision $\mathcal{T}^{(139)}$ contains in its 1-skeleton the union of the two core curves.
\end{theorem}

We will use this to prove our main result about lens spaces.

\begin{named}{\refthm{LensSpaces}}
Let $L(p,q)$ be a lens space, where $p$ and $q$ are coprime integers satisfying $0< q < p$. Let $[a_0, \dots, a_n]$ be the continued fraction expansion of $p/q$ where each $a_i $ is positive. Then there is a universal constant $k_{\mathrm{lens}} >0$ such that
\[ k_{\mathrm{lens}} \sum_{i=0}^n a_i \leq \Delta(L(p,q)) \leq \sum_{i=0}^n a_i. \]
\end{named}

\begin{proof}
The upper bound on $\Delta(L(p,q))$ is fairly straightforward. Indeed, Jaco and Rubinstein \cite{JacoRubinstein:Layered} provide a triangulation with $(\sum_{i=0}^n a_i) - 3$ tetrahedra for $p > 3$ and they conjecture that this is equal to $\Delta(L(p,q))$. For $p =2$ or $3$, the lens space is $\mathbb{RP}^3$ or $L(3,1)$, which both satisfy $\Delta(L(p,q))=2$. Note $\sum_{i=0}^n a_i\geq 2$, so the inequality holds in these cases.

We now focus on the lower bound for $\Delta(L(p,q))$.
The triangulation complexity of $L(4p', 2p' \pm 1)$ was shown by Jaco, Rubinstein and Tillmann \cite{JacoRubinsteinTillmann:Coverings} to be $p'$ for $p' \geq 2$.
So we now assume that the lens space is not of this form and also is not $\mathbb{RP}^3$ or $L(4,1)$.

Let $\mathcal{T}$ be a triangulation of $L(p,q)$.
Our goal is to show that $\Delta(\mathcal{T}) \geq k_{\mathrm{lens}}  \sum_{i=0}^n a_i$ for some universal constant $k_{\mathrm{lens}} > 0$. The approach that we will take is as follows.
We will drill from (a subdivision of) the triangulation $\calT$ two core curves $C$ and $C'$ for the solid tori making up $L(p,q)$, with meridians $\mu$ and $\mu'$. This gives a manifold homeomorphic to $T^2\times [0,1]$. Using the triangulation $\calT$, we will obtain spines for $T^2\times\{0\}$ and $T^2\times\{1\}$ containing $\mu$ and $\mu'$, respectively; these correspond to points on $L(\mu)$ and $L(\mu')$ in the Farey tree. Theorem~\ref{Thm:ProductToriHandles} and Lemma~\ref{Lem:SpineAndShortCurve} then build a path in the Farey tree from $L(\mu)$ to $L(\mu')$, with length bounded by a constant times $\Delta(\calT)$. This must be at least as long as the shortest path from $L(\mu)$ to $L(\mu')$, which is $\sum a_i$. Thus our constructed path will give a bound of the form $\sum a_i \leq (1/k_{\mathrm{lens}})\Delta(\calT)$ for some constant $k_{\mathrm{lens}}$. This completes the outline of the proof of the theorem.

By \refthm{LensSpaceCore}, $\mathcal{T}^{(139)}$ contains in its 1-skeleton the union of the two core curves $C$ and $C'$.
To build our spines on $T^2\times\{0\}$ and $T^2\times\{1\}$, we not only need $C$ and $C'$ to be simplicial, but also a regular neighbourhood of these curves must be simplicial, and a meridian of the simplicial neighbourhood must also be simplicial. This can be obtained by further subdivision. However, as we subdivide, we will need a bound on the length of the simplicial meridian to apply \reflem{SpineAndShortCurve}. Again this can be obtained by careful subdivision, as follows. Beginning with $\calT^{(139)}$,
take a further two barycentric subdivisions, so that a regular neighbourhood $N(C \cup C')$ for $C\cup C'$ is simplicial in $\calT^{(141)}$.
Observe that $\bdy N(C\cup C')$ consists of faces, edges, and vertices that belong to $\calT^{(141)}$ but do not lie in faces, edges and vertices (respectively) of $\calT^{(139)}$. Observe that barycentric subdivision adds a new central vertex to each tetrahedron of $\calT^{(140)}$, and this vertex meets exactly $24$ tetrahedra in $\calT^{(141)}$. Observe also that an edge running from this central vertex to one of the new vertices on a face will meet $30$ tetrahedra: $24$ at one endpoint, and an additional six at the other endpoint.

We may find two curves $\overline C$ and $\overline C'$ that are simplicial on $\bdy N(C\cup C')$, that are parallel copies of $C$ and $C'$, respectively, and that are made up of edges each meeting at most $30$ tetrahedra. We may take a further two barycentric subdivisions, creating $\calT^{(143)}$, so that $N(\overline{C} \cup \overline{C}')$ is simplicial.
Then $N(\overline{C} \cup \overline{C}')$
consists of those simplices in $\calT^{(143)}$ that have non-empty intersection with $\overline C$ and $\overline C'$. We remove the interior of this regular neighbourhood from $L(p,q)$, and thereby obtain a triangulation of $T^2 \times [0,1]$. Because each barycentric subdivision increased the number of tetrahedra by a factor of $24$, this triangulation has complexity at most $(24)^{143} \Delta(\mathcal{T})$.

Consider any vertex $v$ of $\calT^{(141)}$ lying on $\overline C$. The union of the simplices in $\calT^{(143)}$ incident to this vertex is a simplicial 3-ball $B$. The intersection between $\partial B$ and $\overline C$ consists of two points. The union of the simplices in $\partial B$ incident to one of these points is a disc. The boundary of this disc is a meridian curve $\mu$ for $N(\overline C)$. Since at most 30 tetrahedra of $\calT^{(141)}$ are incident to $v$, we deduce that there is a universal upper bound (120 in fact) for the length of $\mu$ in $\calT^{(143)}$. Each vertex of $\mu$ lies in the interior of a 3-simplex or 2-simplex of $\calT^{(141)}$. Hence, it is incident to at most two 3-simplices of $\calT^{(141)}$. These 3-simplices contain at most $2 \times (24)^2$ tetrahedra of $\calT^{(143)}$ and hence at most $12 \times (24)^2 = 6912$ edges of $\calT^{(143)}$. We deduce that each vertex of $\mu$ is incident to at most 6912 edges in $\bdy N(\overline C)$. Similarly, on $\partial N(\overline C')$, there is a meridian curve $\mu'$ with length at most 120 and again with the property that each vertex that it runs through is incident to at most 6912 edges in $\bdy N(\overline C')$.

As in Remark \ref{Rem:AddBoundaryTimesI}, we attach a triangulation of $\bdy N(C\cup C')$ onto the triangulation of $T^2 \times [0,1]$, to form a new triangulation $\calT'$ of $T^2 \times [0,1]$. This satisfies $\Delta(\calT') \leq 
33 (24)^{143} \Delta(\mathcal{T})$. By Lemma \ref{Lem:ComplexityDualHandleStructure}, the dual handle structure $\calH$ is pre-tetrahedral and satisfies $\Delta(\calH) \leq 
33 (24)^{143} \Delta(\mathcal{T})$. As explained in Remark \ref{Rem:NoParallelity}, it has no parallelity handles. In particular, it admits no annular simplifications.

There is a copy of the meridian curve $\mu$ in $\calT'$. This is a sequence of vertices and edges, and hence it corresponds to a sequence of 2-handles and 1-handles in the handle structure of $T^2 \times \{0,1\}$. The boundary of each of these 2-handles has length at most $6912 \times 2$ in the cell structure, and the boundary of each 1-handle has length $4$. The union of these 1-handles and 2-handles is an annulus. Each of the boundary components of the annulus is cellular and has length at most $120 \times ((6912 \times 2) + 4) = 1659360$ in $\calH$. Pick one of these boundary components and extend it to a cellular spine $\Gamma$ in $T^2 \times \{ 0 \}$. Similarly, there is a cellular curve in $T^2 \times \{ 1 \}$ that is parallel to $\mu'$ and that has length at most $1659360$. 

By \refthm{ProductToriHandles} and \reflem{EdgeSwapBound}, there is a universal constant $k_{\mathrm{hand}} > 0$ and a sequence of at most $24k_{\mathrm{hand}}\,\Delta(\mathcal{H})$ edge contractions and expansions taking $\Gamma$ to a cellular spine $\Gamma'$ in $T^2\times\{1\}$. By \reflem{SpineAndShortCurve}, there is a further sequence of at most $24 + 4\cdot 1659360 = 6637464$ edge swaps taking $\Gamma'$ to a spine $\Gamma''$ that contains $\mu'$. By \reflem{EdgeSwapBound}, this is realised by at most $24\cdot 6637464 < 10^8$ edge contractions and expansions. Since $\Gamma$ and $\Gamma''$ contain meridians as subsets, they correspond to vertices in the Farey tree that lie on the lines $L(\mu)$ and $L(\mu')$, as in \refdef{FareyLine}. The distance in the spine graph is exactly twice the distance in the Farey tree, since a 2-2 Pachner move in a triangulation is realised by an edge contraction then expansion.
Hence, the distance between these lines is at most
$24k_{\mathrm{hand}} \, \Delta(\mathcal{H}) + 10^8 \leq 33 (24)^{144}k_{\mathrm{hand}} \, \Delta(\calT) + 10^8$
in the Farey tree. 
By \reflem{LinesInFareyTree},
this distance is $(\sum_{i=0}^n a_i)-1$, where $[a_0, \dots, a_n]$ is the continued fraction expansion of $p/q$.
Hence, $\Delta(L(p,q))$ is at least a linear function of $\sum_{i=0}^n a_i$. Since $\sum_{i=0}^n a_i \geq 1$, the additive part of this linear function can be eliminated, at the possible cost of changing the multiplicative constant. So, $\Delta(L(p,q))$ is at least $k_{\mathrm{lens}} \sum_{i=0}^n a_i$ for some universal $k_{\mathrm{lens}} > 0$.
\end{proof}

%%%%%%%%%%%%%%%%%%%%%%%%%%%%%%%%%%%%%%%%%%%%%%%%%%%%%%%%%%%%%%%%%
\section{Prism manifolds and Platonic manifolds}
\label{Sec:Prism}

We start by considering the prism manifold $P(p,q)$.
This is obtained by gluing together the solid torus and $K^2 \twist I$, the orientable $I$-bundle over the Klein bottle, via a homeomorphism between their boundaries. The resulting manifold is determined, up to homeomorphism, by the slope on the boundary of $K^2 \twist I$ to which a meridian disc of the solid torus is attached. Now, the boundary of $K^2 \twist I$ has a canonical framing, as follows. There are only two non-separating 
simple closed curves $\lambda$ and $\mu$ on the Klein bottle, where $\lambda$ is orientation-reversing and $\mu$ is orientation-preserving. The inverse images of these in the boundary of $K^2 \twist I$ are curves with slopes $\tilde \lambda$ and $\tilde \mu$. The prism manifold $P(p,q)$ is obtained by attaching the meridian disc of the solid torus along a curve with slope $p \tilde \lambda + q \tilde \mu$, when these slopes are given some choice of orientation.

\begin{named}{\refthm{Prism}}
Let $p$ and $q$ be non-zero coprime integers and let $[a_0, \dots, a_n]$ denote the continued fraction expansion of $p/q$ where $a_i$ is positive for each $i > 0$. Then, $\Delta(P(p,q))$ is, to within a universally bounded multiplicative error, equal to $\sum_{i=0}^n a_i$.
\end{named}

\begin{proof} As usual, the upper bound on $\Delta(P(p,q))$ is fairly straightforward. Start with a 1-vertex triangulation of the Klein bottle in which $\lambda$ and $\mu$ are edges. The inverse image in $K^2 \twist I$ of the three edges in $K^2$  is three squares, which we can triangulate using two triangles each. If we cut $K^2 \twist I$ along these three squares, the result is two prisms, which can be triangulated using eight tetrahedra. Now attach onto the boundary of $K^2 \twist I$ some tetrahedra, so that the resulting boundary has a 1-vertex triangulation, and where two of its edges have slopes $\tilde \lambda$ and $\tilde \mu$. Now apply 2-2 Pachner moves taking this triangulation to one that includes $p/q$ as an edge. Then glue on a triangulation of the solid torus with boundary triangulation containing a meridian as an edge. Thus the resulting number of tetrahedra in the triangulation of $P(p,q)$ is at most $\sum_{i=0}^n a_i$ plus a constant.

For the lower bound, note that $P(p,q)$ is double covered by a lens space $L$. Hence, $\Delta(P(p,q)) \geq \Delta(L)/2$. The inverse image of $K^2 \twist I$ in this double cover is a copy of $T^2 \times [0,1]$. The inverse image of the solid torus is two solid tori, one attached along the slope $p/q$ and the other attached along the slope $-p/q$. According to \refthm{LensSpaces} and \reflem{LinesInFareyTree}, $\Delta(L)$ is at least a constant times the distance in the Farey graph between the lines $L(p/q)$ and $L(-p/q)$. To compute this distance, consider the hyperbolic geodesic joining $p/q$ to $-p/q$. This is divided into two half-infinite geodesics by the imaginary axis. As we travel along one of these geodesics, starting at the imaginary axis and ending at $p/q$, we recover the splitting sequence for $p/q$. Hence, this corresponds to a path in the Farey tree with length $\sum_{i=0}^n a_i$. The path running from the imaginary axis to $L(-p/q)$ has the same length. Hence, $\Delta(L)$ is at least a constant times  $\sum_{i=0}^n a_i$.
\end{proof}

A Platonic manifold is an elliptic manifold $M$ that admits a Seifert fibration with base space that is the quotient of the 2-sphere by the orientation-preserving symmetry group of a Platonic solid. In other words, the base space $\Sigma$ is a 2-sphere with three cone points of orders $(2,3,3)$, $(2,3,4)$ or $(2,3,5)$. 

If we remove the three singular fibres from $M$, the result is a circle bundle over the three-holed sphere. Thus, $M$ is obtained from this circle bundle by attaching three solid tori. For convenience, we also remove one regular fibre, and the resulting manifold $M'$ is a circle bundle over the four-holed sphere $S$. Now orientable circle bundles over orientable surfaces with non-empty boundary are trivial. Thus, $M'$ is just a copy of $S \times S^1$. We fix a meridian and longitude for each boundary component of $S \times S^1$, by declaring the longitude to be of the form $C \times \{\ast\}$ for the relevant boundary component $C$ of $S$, and declaring the meridian to be $\{\ast\} \times S^1$. Thus, the Dehn filling slopes are given by four fractions $p_0/q_0$, $p_1/q_1$, $p_2/q_2$ and $p_3/q_3$, where $q_0  = 1$, $q_1 = 2$, $q_2 = 3$ and $q_3 = 3$, $4$ or $5$. The Euler number is just the sum $p_0/q_0 + p_1/q_1 + p_2/q_2 + p_3/q_3$. Without changing the manifold or its Seifert fibration, we can adjust these slopes by adding an integer to one and subtracting an integer from another. In this way, we can arrange $p_1/q_1$, $p_2/q_2$ and $p_3/q_3$ all to lie strictly between $0$ and $1$. 

\begin{named}{\refthm{Platonic}}
Let $M$ be a Platonic elliptic 3-manifold, and let $e$ denote the Euler number of its Seifert fibration. Then, to within a universally bounded multiplicative error, $\Delta(M)$ is $|e|$.
\end{named}

\begin{proof}
The upper bound is straightforward. We can form a triangulation of $S \times S^1$ with a fixed number of tetrahedra, and where the longitudes and meridians are all simplicial. We can also arrange that the triangulation of each boundary component has a single vertex. Since $p_1/q_1$, $p_2/q_2$ and $p_3/q_3$ take only finitely many possible values, we can attach triangulated solid tori so that the meridian disc is attached to these slopes, using a universally bounded number of tetrahedra. The final slope $p_0/1$ is integral.  Hence, using at most $|p_0|$ many 2-2 Pachner moves, we many arrange that this slope $p_0/1$ is simplicial. We can then attach a triangulated solid torus to form a triangulation of $M$. The difference between $p_0$ and the Euler number $e$ is bounded above by 3, since $p_1/q_1$, $p_2/q_2$ and $p_3/q_3$ all lie between $0$ and $1$. So, the number of tetrahedra is at most $|e| + c$ for some universal constant $c$. This is at most a multiple of $|e|$ as $e$ cannot be zero, since $M$ would then contain an embedded non-separating orientable surface, which is impossible in a rational homology 3-sphere.

We now establish the lower bound on $\Delta(M)$.  The Seifert fibration $M \rightarrow \Sigma$ induces a surjective homomorphism $\pi_1(M) \rightarrow \pi_1(\Sigma)$, where the latter group is the orbifold fundamental group of the base space $\Sigma$. The kernel of this homomorphism has index at most $60$. Let $\tilde M$ be the corresponding finite cover of $M$. Then $\Delta(\tilde M) \leq 60 \Delta(M)$. The Seifert fibration on $M$ lifts to a Seifert fibration on $\tilde M$. The Euler number $\tilde e$ of $\tilde M$ is related to the Euler number of $M$ as follows, using \cite[Theorem 3.6]{Scott}. If $d_1$ is the degree of the covering between the base orbifolds, and $d_2$ is the degree of the coverings between regular fibres, then $\tilde e = ed_1/d_2$. In particular, $|\tilde e| \geq |e|/60$. The Seifert fibration of $\tilde M$ has base space a 2-sphere and has no singular fibres, and therefore $\tilde M$ is the lens space $L(\tilde e, 1)$. Therefore, by \refthm{LensSpaces},
\[ \Delta(M) \geq \Delta(\tilde M)/60 \geq (k_{\mathrm{lens}}/60) |\tilde e| \geq (k_{\mathrm{lens}}/3600) |e|. \qedhere\]
\end{proof}

\bibliographystyle{amsplain}
\bibliography{references-complexitylens}

\end{document}